\documentclass[onefignum,onetabnum]{siamart190516}

\usepackage{amsmath, amsfonts,amssymb,color}
\usepackage[english]{babel}

\usepackage{dsfont}
\usepackage{mathabx}
\usepackage{enumitem}
\usepackage{graphicx,epsfig}
\usepackage{float}
\usepackage{capt-of}
\usepackage{tikz}

\setlist[description]{leftmargin=0cm,labelindent=0cm}

\newcommand{\YA}[1]{{{#1}}}

\newcommand{\R}{{\mathbb R}}
\newcommand{\N}{{\mathbb N}}

\newcommand{\Z}{{\mathbb Z}}
\newcommand{\T}{{\mathbb T}}

\newcommand{\cM}{{\mathcal M}}

\newcommand{\HoneT}{\overline{H}_{1,T}}
\newcommand{\HoneplusT}{\overline{H}_{1,+,T}}
\newcommand{\HoneminusT}{\overline{H}_{1,-,T}}

\newcommand{\chiperp}{\chi_{{\rm per},p}}

\newcommand{\per}{{\rm per}}

\newcommand{\ds}{\displaystyle}

\definecolor{mypurple}{RGB}{140,0,255}
\definecolor{myred}{rgb}{255,0,0}
\definecolor{mydarkturquoise}{RGB}{0,206,209}
\definecolor{mydeeppink}{RGB}{255,20,147}
\definecolor{darkblue}{RGB}{0,0,140}
\definecolor{blue2}{RGB}{0,0,0}
\definecolor{middleblue}{RGB}{0,0,71}
\definecolor{light-gray}{gray}{0.9}
\definecolor{ProcessBlue}{cmyk}{1,0,0,0.40}
\definecolor{Black}{cmyk}{0,0,0,1}
\definecolor{Red}{cmyk}{0,1,1,0.2}
\definecolor{Green}{cmyk}{0.9,0,1,0}
\definecolor{Orange}{cmyk}{0,0.61,0.87,0.5}
\definecolor{Fuchsia}{cmyk}{0.47,0.91,0,0.06}
\definecolor{PineGreen}{cmyk}{0.92,0,0.59,0.30}

\usetikzlibrary{patterns}

\newlength{\flexcheckerboardsize}

\newcommand{\defineflexcheckerboard}[4]{
    \setlength{\flexcheckerboardsize}{#2}
    \pgfdeclarepatterninherentlycolored{#1}
        {\pgfpointorigin}{\pgfqpoint{2\flexcheckerboardsize}    
        {2\flexcheckerboardsize}}
        {\pgfqpoint{2\flexcheckerboardsize}
        {2\flexcheckerboardsize}}%
        {
            \pgfsetfillcolor{#4}
            \pgfpathrectangle{\pgfpointorigin}{
            \pgfqpoint{2.1\flexcheckerboardsize}    
                {2.1\flexcheckerboardsize}}
          \pgfusepath{fill}
          \pgfsetfillcolor{#3}
          \pgfpathrectangle{\pgfpointorigin}
            {\pgfqpoint{\flexcheckerboardsize}
            {\flexcheckerboardsize}}
          \pgfpathrectangle{\pgfqpoint{\flexcheckerboardsize}
            {\flexcheckerboardsize}}
            {\pgfqpoint{\flexcheckerboardsize}
            {\flexcheckerboardsize}}
            \pgfusepath{fill}
        }
}

\defineflexcheckerboard{flexcheckerboard_g}{5mm}{black!50}{black!10}

\defineflexcheckerboard{flexcheckerboard_bw}{.5mm}{black}{white}
\defineflexcheckerboard{flexcheckerboard_redblue}{.5mm}{red}{blue}
\defineflexcheckerboard{flexcheckerboard_greenorange}{1mm}{green}{orange}
\defineflexcheckerboard{flexcheckerboard_bluecyan}{.2mm}{cyan}{blue}

\numberwithin{equation}{section}
 \def\theequation{\thesection.\arabic{equation}}
 \newtheorem{remark}{Remark}[section]

\DeclareMathOperator*{\argmin}{arg\,min}

\begin{document}

\title{ Homogenization of \\Hamilton-Jacobi equations with  defects\\leading to stratified problems}

\author{
   Yves Achdou\thanks { Universit{\'e} Paris Cit{\'e} and  Sorbonne Universit{\'e}, CNRS, Laboratoire Jacques-Louis Lions, (LJLL), F-75006 Paris, France, achdou@ljll-univ-paris-diderot.fr}
   \and
    Claude Le Bris\thanks{{\'E}cole des Ponts and INRIA, 6-8 avenue Blaise Pascal, Cit{\'e} Descartes, Champs-sur-Marne, 77455 Marne La Vall{\'e}e, France,   claude.le-bris@enpc.fr}
  }

\maketitle
\begin{abstract}
  We study homogenization  of a class of bidimensional stationary Hamilton-Jacobi equations where the  Hamiltonian is obtained by perturbing near a half-line of the state space a Hamiltonian that either does not have fast variations with respect to the state variable, or depends on the latter in a periodic manner. We prove that the limiting problem belongs to the class of stratified problems introduced by A. Bressan and Y. Hong and later studied by G. Barles and E. Chasseigne.
  The related Whitney stratification is made of a submanifold of dimension zero, namely the endpoint of the half-line,  a submanifold of dimension one, the open half-line, and the complement of the latter two sets which is a submanifold of dimension two. The limiting problem involves effective Hamiltonians that are associated to the  above mentioned three submanifolds and keep track of the perturbation. Another example in which the Hamiltonian is perturbed in a tubular neighborhood of a line is studied.
\end{abstract}

\section{Introduction}
\label{sec:introduction}

Recently, there has a been a growing research effort on Hamil- -ton-Jacobi equations with discontinuous Hamiltonians and optimal control problems with discontinuities in the cost or the dynamics. In particular, important progress have been made when the discontinuities are located on submanifolds of codimension one, see the first part of the book of G. Barles and E. Chasseigne, \cite{MR4704058}, and the references therein.  The case when the locus of the  discontinuities may be locally of any codimension is more complex.
In such situations, a complete theory has been developed under the important geometric assumption  that the problem is {\sl stratified}, i.e. that the discontinuities lie on a union of submanifolds that form a Whitney stratification of $\R^d$. A Whitney stratification is a partition of  $\R^d$  into a  family $(\cM_k)_{k=0,\dots, d}$ of disjoint submanifolds  with specific properties  (in particular, $\cM_k$  is of dimension $k$  if it is not empty). Since the stratifications that will arise in the present work are  simple, we do not thoroughly write the  rather long definition of Whitney stratifications and  refer the reader to  \cite[Definition 2.2]{MR4704058}. Besides, the stratifications that will appear below are {\sl flat} Whitney stratifications, which means that the connected components of $\cM_k$ locally coincide  with affine subspaces of dimension $k$,  see \cite[Definition 2.3]{MR4704058}.  Stratified problems have been introduced by A. Bressan and Y. Hong, see \cite{MR2291823} and later studied by G. Barles and E. Chasseigne, see \cite{MR4704058}. In particular, these authors have proved  comparison principles for viscosity sub/supersolutions of stratified problems. These results are of  course crucial because they imply  uniqueness of viscosity solutions.
Note that  in the theory of stratified viscosity solutions, the convexity of the Hamiltonian with respect to the momentum is a key assumption.

In the present work, we aim at giving examples of situations in which the homogenization of a continuous problem leads to a stratified problem.

In that respect, this work can be seen as the continuation of \cite{MR4643677}, where the authors studied homogenization  of a class of stationary Hamilton-Jacobi equations
with a Hamiltonian  obtained by perturbing a periodic Hamiltonian in a bounded neighborhood of  the origin. In  \cite{MR4643677},  such a perturbation was termed  a local \emph{defect}. The main result of \cite{MR4643677} is  that the limiting problem consists of a Hamilton-Jacobi equation outside the origin, with the same effective Hamiltonian as in periodic homogenization, supplemented at the origin with an effective  Dirichlet condition  that keeps track of the perturbation. 
Note  that the codimension of the origin is $d$ and that the Dirichlet problems in  $\R^d\setminus \{0\}$ may have several viscosity solutions in the sense of H. Ishii, see e.g. \cite[chapter V, section 4]{MR1484411}. However, it was shown in \cite{MR4643677} that the effective  problem is a  {\sl stratified} problem in the sense of  \cite{MR4704058}, for which uniqueness holds. In other words, the formulation of the effective Dirichlet problem found in \cite{MR4643677} is more precise than the formulation of H. Ishii.  Besides, according to the effective  Dirichlet data, the
solution of the effective problem  may or may not coincide with the solution of the Hamilton-Jacobi equation posed in the whole space.  Finally,  \cite{MR4643677} is  much related to  a series of lectures~\cite{PLL-college} at Coll{\`e}ge de  France by P-L. Lions, where some results obtained in collaboration with P. Souganidis, ~\cite{PLL}, were described.  In the control theoretic interpretation of \cite{PLL-college,PLL}, the typical local perturbation of the running cost was a bump oriented so that  the neighborhood of the origin became repulsive. In~\cite{PLL-college}, it was shown that the presence of such a  defect indeed does not affect the homogenized limit, but only possibly \emph{"the next order correction"}, that is the definition of the corrector function itself. In other words, as mentioned above, homogenized limit is a viscosity solution of the effective Hamilton-Jacobi equation posed in the whole space. In contrast,  it was shown in \cite{MR4643677} that if the defect makes the origin attractive, then it  affects the homogenized limit itself, and
the latter is impacted by the effective Dirichlet condition. We will discuss the  technical aspects of  \cite{MR4643677} in more details later, when they are needed.

Note that homogenization theory in the presence of local defects within an otherwise periodic environment was first introduced in~\cite{Milan}, in the first of a series of works by X.~Blanc, C. Le Bris and P-L.~Lions. It was further developed in~\cite{CPDE-2015,CPDE-2018} and other subsequent works by various authors, considering different classes of defects such as, in particular, interfaces between two different periodic media.  In those articles, the typical setting  is that of a linear non-degenerate elliptic equation, first in divergence form and next in more general form. Only recently, some quasilinear second order elliptic equation was considered in~\cite{MR4523487}.

In the present work, we aim at studying examples in which homogenization leads to stratified problems with  more complex stratifications than in \cite{MR4643677}, for example with a complete hierarchy of submanifolds. The new technical difficulty will consist of  combining the different correctors associated with the different submanifolds.

 We work in $\R^2$ for simplicity, and  discuss  the homogenization limit for a first order stationary Hamilton-Jacobi equation of the form 
\begin{equation*}
\alpha \, u_\varepsilon+ H\left(x,\frac x \varepsilon, D u_\varepsilon\right)=0\quad \hbox{ in }\R^2.
\end{equation*}
 The Hamiltonian is a continuous function of its three arguments, convex with respect to its last argument, and arises from an optimal control problem (the assumptions will be made more precise later).
We mainly consider three situations:
\begin{description}
\item[Case 1] The Hamiltonian $ H(x,y,p)$ does not depend on the fast variable $y\in \R^2$,   except if $y\in \Omega$, where  $\Omega$ is a  neighborhood of the line half-line $\overline{\cM_1}$, $\cM_1= \{(y_1,0), y_1<0\}$. The set $\Omega$ is connected and is obtained as the union of a tubular
  neighborhood of $\cM_1$ and of a bounded neighborhood of  the endpoint of $\cM_1$, namely the origin $0_{\R^2}$. 
 It is also assumed that for $y \in \Omega$ such that $|y_1|$ is large enough, $H(x, y,p)$ coincides with $H_{1,\per}(x,y,p)$, where $ H_{1,\per}$ is periodic with respect to $y_1$, the first coordinate of $y$. The  regions of $\R^2$ introduced above are displayed on   Figure \ref{fig:1}. 
\item[Case 2] The function $y\mapsto H(x,y,p)$ is $1$-periodic with respect to both fast variables $y_1$ and $y_2$, the two coordinates of $y$,  except in the region $\Omega$  introduced above.  For $y\in \Omega$  such that $|y_1|$ is large enough,  $H(x, y,p)= H_{1,\per}(x,y,p)$, where $ H_{1,\per}$ is $1$-periodic with respect to $y_1\in \R$, see Figure \ref{fig:2} for an illustration.
  \item[Case 3] Here, we set $\cM_1= \cM_{1,-} \cup\cM_{1,+}$ where  $\cM_{1,\pm}=  \{(y_1,0), \pm y_1>0\}$. The function $\R^2 \ni y\mapsto H(x,y,p)$ is constant  except in a  tubular neighborhood $\Omega$ of  the straight line $y_1=0$, see Figure \ref{fig:3} for an illustration. 
  For $y$ near $\cM_{1,-}$ and far enough from the origin, $H(x, y,p)$ coincides with $ H_{1,-,\per}(x,y,p)$, where $ H_{1,-,\per}$ is periodic with respect to $y_1\in \R$. Similarly, near $\cM_{1,+}$ and far enough from the origin, $H(x, y,p)= H_{1,+,\per}(x,y,p)$, where $ H_{1,+,\per}$ is periodic with respect to $y_1$. Note that $ H_{1,+,\per}$ and  $H_{1,-,\per}$ can be chosen completely independently, with different periods  for example.  
\end{description}

In the three cases described above,  $H$ is obtained by perturbing in $\Omega$ a function either invariant or periodic  with respect to the fast variable $y$. We can see this perturbation as a  \emph{ longitudinal defect} localized near a half-line or a line.

The present work can also be seen as the continuation of  two articles  \cite{MR3565416} and \cite{MR3912640}, that addressed   Hamilton-Jacobi equations  in an environment consisting of two different homogeneous media separated by an oscillatory interface. The oscillations of the interface have small period and amplitude, and as  the latter parameter vanishes,
the interface tends to an hyperplane. At the limit when both parameters vanish, one finds on the flat interface an effective nonlinear  transmission condition  keeping memory of the previously mentioned microscopic oscillations. The effective problem can be seen as a stratified problem in which the discontinuity lies on  a submanifold of codimension one. By contrast, in the present work, the effective problem will involve both a submanifold of codimension one, namely $\cM_1$ defined above and a submanifold of codimension two, namely $\cM_0=\{0\}$.
Key arguments in \cite{MR3565416, MR3912640} are the construction of families of correctors that account for the localized perturbations of the environment and are defined in unbounded domains. We will see that the construction of such correctors  also plays a key role in the present work. 

Our main results (Theorem \ref{sec:main-result-1} for Case 1,  Theorem \ref{sec:main-result-2} for Case 2 and Theorem \ref{sec:main-result-3} for Case 3) 
state that when $\varepsilon$ tends to zero, $u_\varepsilon$ converges to the solution of an effective stratified problem.
In the three examples described above, as in \cite{MR4643677} and in contrast with  \cite{MR3565416,MR3912640},  we see that the origin plays a special role, because near this point,  the perturbation of the Hamiltonian is not periodic   with respect to the longitudinal variable. This explains the fact that the effective problem will be a stratified problem for the flat Whitney stratification of $\R^2$ : $\R^2= \cM_0\cup\cM_1\cup \cM_2$, where the one-dimensional submanifold $\cM_{1}$ has been defined above (the definition in Case 3 differs from that in Cases 1-2), $\cM_0=\{0\}$ is of dimension zero and $\cM_2= \R^2 \setminus ( \cM_0\cup\cM_1)$ is an open subset of $\R^2$. It will involve both effective  (tangential) Hamiltonians defined on $\cM_1$ and reminiscent of those found in \cite{MR3565416}, and an effective Dirichlet data at the origin, similar to the one found in \cite{MR4643677}.

Note that we have decided to focus on bidimensional problems only for simplicity. Generalization to the homogenization of  Hamilton-Jacobi equations in $\R^d$ with defects localized near a half-hyperplane or two aligned complementary half-hyperplanes does not bring any new difficulty. Similarly,  homogenization with defects located near a smooth connected submanifold of dimension $m<d$ of $\R^d$  may be tackled with the same techniques as those proposed below.

Note also that the assumption regarding the continuity of $H$ with respect to $y$ may be relaxed: for example, our techniques can be applied to the homogenization  of a stratified problem with a Hamiltonian  piecewise constant in the state variable, and associated with the Whitney stratification $\R^2= \widetilde \cM_{2,\varepsilon}\cup  \widetilde \cM_{1,\varepsilon}\cup \widetilde \cM_0$, where $\widetilde \cM_0=\{0\}$, $\widetilde \cM_{1,\varepsilon}=\{(x_1,\varepsilon \gamma(x_1/\varepsilon)), x_1<0\}$, $\gamma$ being a smooth and periodic function defined on $\R$
and such that $\gamma(0)=0$, and $\widetilde \cM_{2,\varepsilon}=  \R^2 \setminus ( \widetilde\cM_0\cup\widetilde\cM_{1,\varepsilon})$. In this case, the effective problem is a stratified problem for the flat Whitney stratification  $\R^2= \cM_0\cup\cM_1\cup \cM_2$, with $\cM_1$ defined in Case 1 and $\cM_0=\{0\}$.

The article is organized as follows: in Section \ref{sec:first_example}, focusing on Case 1, we set the problem, state the main result (Theorem \ref{sec:main-result-1}) and review relevant notions belonging to the theory of stratified control problems. Section  \ref{sec:proof} is devoted to the proof of Theorem \ref{sec:main-result-1}.
Cases 2 and 3  are dealt with  in Section \ref{sec:generalization},  more rapidly, because many arguments are similar.

 \section{A prototypical case with a  longitudinally periodic  defect  localized in the neighborhood of a half-line}
\label{sec:first_example}

\subsection{Setting and assumptions}
\label{sec:setting}
Let us define the problem and list the assumptions.
We wish to emphasize  that  these assumptions  (regarding controllability, convexity, coercivity, regularity, - see below) are classical and that we do not seek to make them as general as possible, our focus being on other issues.

The description that follows is illustrated on Figure \ref{fig:1} below.


Hereafter, $B_r(x)$ stands for the ball of radius $r$ centered at $x\in \R^2$.  Given $R_0>0$, let $\Omega$ be the subset of $\R^2$ defined by
\begin{eqnarray}
  \label{eq:1}
  \Omega=\Bigl((-\infty,0]\times (-R_0,R_0)\Bigr) \cup B_{R_0}(0).
\end{eqnarray}
For a small positive parameter  $\varepsilon$
that will eventually vanish, set
$\Omega_\varepsilon=\varepsilon \Omega$.

We consider Hamilton-Jacobi equations related to infinite horizon optimal control problems in $\R^2$. 
The Hamiltonian $H_\varepsilon: \R^2\times \R^2 \to \R$ is of the form 
\begin{equation}
  \label{eq:2}
H_\varepsilon(x,p)= \max_{a\in A } \Bigl(-p\cdot f_\varepsilon (x,a)-\ell_\varepsilon(x,a)\Bigr).
\end{equation}
Here, $A$ is a compact metric space. A partial justification of this assumption
is that it is made in \cite{MR2291823,MR4704058}, although it does not seem crucial in the theory of stratified control problems.

\medskip

We assume that the  dynamics  $f_\varepsilon: \R^2\times A\to \R^2$ and cost  $\ell_\varepsilon: \R^2\times A\to \R$ have the following properties:
\begin{enumerate}
\item They  have the form \begin{equation}
  \label{eq:3}
  f_\varepsilon(x,a)= f\left( x,\frac x \varepsilon, a\right),\quad \hbox{and} \quad  \ell_\varepsilon(x,a)= \ell\left( x,\frac x \varepsilon, a\right),
\end{equation}
where $f: \R^2\times\R^2\times A \to \R^2$ and   $\ell: \R^2\times\R^2\times A \to \R$ are  bounded and  continuous. The function $f$ is assumed
 Lipschitz continuous with respect to its  first two variables uniformly with respect to its third variable, i.e. 
 for any  $(x,y,\tilde x, \tilde y)\in (\R^2)^4$ and $a\in A$,
 \begin{equation}
   \label{eq:flip}
    |f(x,y,a)-f(\tilde x,\tilde y,a)|\le   L_f (|x-\tilde x| +|y-\tilde y|).
 \end{equation}
    Concerning $\ell$, there exists  a modulus of continuity $\omega_\ell$ such that 
 for any  $(x,y,\tilde x, \tilde y)\in (\R^2)^4$ and $a\in A$,
 \begin{equation}
   \label{eq:l_UC}
    |\ell(x,y,a)-\ell(\tilde x,\tilde y,a)|\le    \omega_\ell(|x-\tilde x| +|y-\tilde y|).
 \end{equation}
  Define
  \begin{equation}
    \label{eq:MfMl}
M_f=     \sup_{x\in \R^2, y\in \R^2, a \in    A} | f(x,y,a)|, \quad M_\ell=     \sup_{x\in \R^2, y\in \R^2, a \in    A} | \ell(x,y,a)| .       
  \end{equation}
  We also suppose that there exists  $r_f>0$ such that for any $x\in \R^2$, $y\in \R^2$, $\{f(x,y,a),a\in A\}$ contains the ball $B_{r_f}(0)$, which implies that the trajectories are locally strongly controllable.
 
\item    There exist functions $\bar{f}:  \R^2\times A\to \R^2$ and $\bar{\ell}:  \R^2\times A\to \R$ such that
  \begin{equation}
  \label{eq:4}
 f( x,y, a)= \bar{f}(x,a),\quad \hbox{and} \quad  \ell( x,y, a)= \bar{\ell}(x,a) \quad \hbox{if } y\not \in \Omega.
\end{equation}

\item
  There exist functions $f_{1,\per}:  \R^2\times \R^2\times A\to \R^2$, $(x,y,a)\mapsto f_{1,\per}(x,y,a)$ and  $\ell_{1,\per}:  \R^2\times \R^2\times A\to \R$, $(x,y,a)\mapsto \ell_{1,\per}(x,y,a)$ such that
  \begin{enumerate}
     \item $f$ and $\ell$ coincide respectively  with $f_{1,\per}$ and $\ell_{1,\per}$ in the set $\R^2 \times \Bigl( (-\infty,0] \times \R \Bigr) \times A$, i.e. for all $x\in \R^2$, $y\in  (-\infty,0] \times \R$ and $a\in A$, $f(x,y,a)=f_{1,\per}(x,y,a)$ and $\ell(x,y,a)=\ell_{1,\per}(x,y,a)$
  \item $f_{1,\per}$ and  $\ell_{1,\per}$ are $1$-periodic with respect to $y_1$
  \item for all $x\in \R^2$ and $a\in A$, $f_{1,\per}(x,y,a)=\bar f (x,a)$ and $\ell_{1,\per}(x,y,a)=\bar \ell (x,a)$ if $|y_2|\ge R_0$.
  \end{enumerate}
  The role of the index $1$ in  $f_{1,\per}$ and $\ell_{1,\per}$   is to emphasize that the functions are periodic with respect to the first coordinate $y_1$ of $y$. 
\end{enumerate}

\begin{remark}
  In Subsection \ref{sec:generalization1}, we will discuss  the
   extension to the case when $f( x,y, a)$ and  $\ell( x,y, a)$  coincide respectively with  $f_{\per}(x,y,a)$ and $\ell_{\per}(x,y,a)$ at all $ y\not \in \Omega$, where $f_{\per}$ and  $\ell_{\per}$ are $1$-periodic with respect to $y_1$ and $y_2$.
\end{remark}

\medskip

It is easy to check that the Hamiltonian $H_\varepsilon$ defined in \eqref{eq:2} has the following properties:
\begin{enumerate}
\item $H_\varepsilon$ is of the form
\begin{equation}
  \label{eq:7}
  H_\varepsilon(x,p)= H\left( x,\frac x \varepsilon, p\right),
\end{equation}
where $H: \R^2\times\R^2\times \R^2 \to \R$ is defined by $H(x,y,p)= \max_{a\in A } \Bigl(-p\cdot f(x,y,a)-\ell(x,y,a)\Bigr)$. The function $H$ is 
 is convex with respect to its third argument,
and  for any  $x,y,\tilde x,\tilde y, p,q\in \R^2$,
 \begin{eqnarray}
   \label{eq:8} H(x,y,p)\ge r_f |p| - M_\ell,\\
   \label{eq:9}|H(x,y,p)-H(\tilde x, \tilde y,p)| \le  \left (
   \begin{array}[c]{l}
        L_f|p| (|x-\tilde x| + |y-\tilde y|) \\ + \omega_{\ell} (|x-\tilde x| +|y-\tilde y|)
   \end{array}
             \right),                                       \\
   \label{eq:10} |H(x,y,p)-H(x,y,q)|\le M_f |p-q|.
 \end{eqnarray}
 Property (\ref{eq:8}) implies the coercivity of $H$ w.r.t. its third variable uniformly with respect to  its first two variables, i.e.  $\lim_{|p|\to \infty} \inf_{x\in \R^2,y\in \R^ 2} H(x,y,p)=+\infty$.
\item Defining $   \overline H(x,p) =  \max_{a\in A } \Bigl(-p\cdot \bar{f}(x,a)-\bar{\ell}(x,a)\Bigr)$   for $x\in \R^2$ and $p\in \R^2$, we see that
  \begin{equation}
  \label{eq:11}
 H( x,y, p)= \overline H(x,p) \quad \hbox{if } y\not \in \Omega.
\end{equation}
\item  
  Defining $H_{1,\per}:  \R^2\times \R^2\times \R^2\to \R$ by
  \begin{displaymath}
  H_{1,\per}(x,y,p)=\max_{a\in A}  \Bigl(-p\cdot f_{1,\per}(x,y,a)-\ell_{1,\per}(x,y,a) \Bigr),    
  \end{displaymath} we see that
  \begin{enumerate}
     \item $H$ coincides with $H_{1,\per}$ in the set $\R^2 \times \Bigl( (-\infty,0] \times \R \Bigr) \times \R^2$, i.e. for all $x\in \R^2$, $y\in  (-\infty,0] \times \R$ and $p\in \R^2$, $H(x,y,p)=H_{1,\per}(x,y,p)$
  \item $H_{1,\per}$ is $1$-periodic with respect to $y_1$
  \item for all $x\in \R^2$ and $p\in \R^2$, $H_{1,\per}(x,y,p)=\overline H  (x,p)$ if $|y_2|>R_0$.
  \end{enumerate}
\end{enumerate}

\bigskip

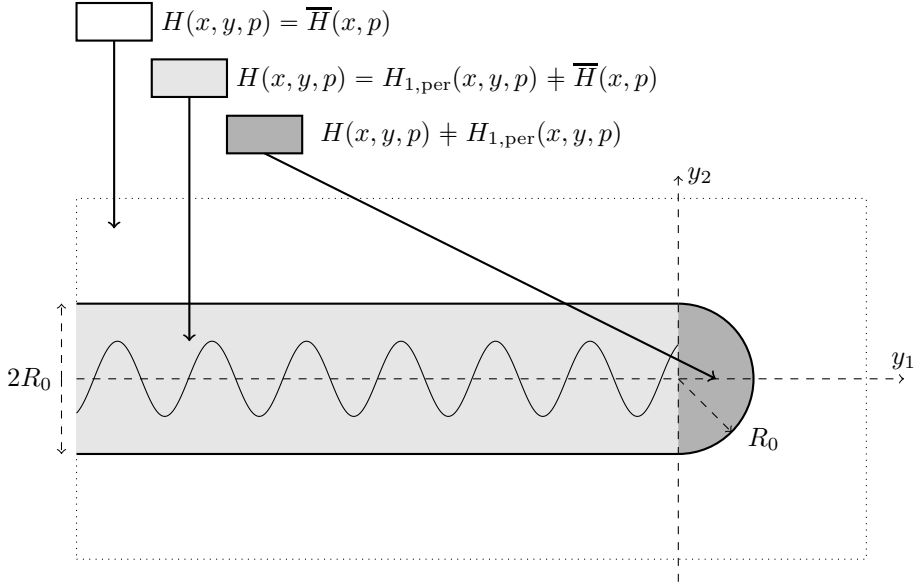
\begin{figure}[H]
  \begin{center}
    \begin{tikzpicture}
     
  \draw [dotted] (-4,-2.4) rectangle (6.5,2.4);
  \filldraw [fill=black!10,black!10] (-4,-1) rectangle (4,1);
  \draw [line width=0.8pt]  (-4,-1) -- (4,-1);
  \draw [line width=0.8pt]  (-4,1) -- (4,1);
  
  \path[fill=black!30] (4,-1)  arc (-90:90:1)  -- (4,-1);
  \path[draw,line width=0.8pt] (4,-1)  arc (-90:90:1);

  \draw [line width=0.8pt]  (-4,5) rectangle (-3,4.5);
   \draw (-3,4.75) node[right] {{$H(x,y,p)=\overline{H}(x,p)$}};
   \filldraw [line width=0.8pt,fill=black!10]  (-3,4.25) rectangle (-2,3.75);
       \draw (-2,4) node[right] {{$H(x,y,p)= H_{1,\per}(x,y,p)\not= \overline{H}(x,p) $}};
       \filldraw [line width=0.8pt,fill=black!30]  (-2,3.5) rectangle (-1,3);
        \draw (-1,3.25) node[right] {{ $ H(x,y,p)  \not=  H_{1,\per}(x,y,p) $}};

        \draw[->, draw opacity =1,line width=0.8pt] (-3.5,4.5) --(-3.5,2);
         \draw[->,draw opacity =1, line width=0.8pt] (-2.5,3.75) --(-2.5,0.5);
  \draw[->, draw opacity =1, line width=0.8pt] (-1.5,3) --(4.5,0);

  \draw[->, dashed] (4,-2.7) --(4,2.7);
  \draw[->,dashed] (-4,0) --(7,0);
  \draw (4,2.7) node[right] {{$y_2$}};
  \draw (7,0) node[above] {{$y_1$}};
  \draw[->, dashed] (-4.2,0) --(-4.2,1);
  \draw[->, dashed] (-4.2,0) --(-4.2,-1);
  \draw[->,dashed] (4,0)--( 4.7071067,-0.7071067) ;
  \draw ( 4.8,-0.8)  node[right] {{$R_0$}};
  \draw (-4.2,0) node[left] {{$2R_0$}};

  \draw [domain=-4:4, samples=10000] plot (\x,{0.5*sin(5*\x r)});
\end{tikzpicture}
\end{center}
\caption{The generic situation described in paragraph \ref{sec:setting}: fixing $(x,p)\in \R^2\times \R^2$, the function $y\mapsto H(x,y,p)$ is continuous, constant in the white region and $1$-periodic with respect to $y_1$ in the light grey region. The set $\overline \Omega$ is the union of the grey regions. The sinusoidal graph is just meant to symbolize the fact that $y\mapsto H(x,y,p)$ is periodic with respect to $y_1$ in the lighter grey region. Note also that $H_{1,\per}$  is defined in the full space $(\R^2)^3$ and that $H_{1,\per}(x,y,p)= \overline{H}(x,p)$ if $|y_2|\ge  R_0$.}
  \label{fig:1}
\end{figure}


\begin{remark}
  Examples fullfilling the assumptions above can be constructed by suitably choosing $f$ and $\ell$  in the additive form 
  $f(x,y,a)=\bar f (x,a) + f_1(y, a)$, $\ell(x,y,a)=\bar \ell (x,a) + \ell_1(y, a)$, where $\bar f: \R^2\times A\to \R^2$, $\bar \ell : \R^2\times A\to \R$,
  $f_1:  \R^2\times A\to \R^2$ and $\ell_1: \R^2\times A\to \R$ are smooth,
$f_1$ and $\ell_1$  vanish outside $\overline \Omega \times A$ and coincide on $\Bigl( (-\infty,0]\times \R\Bigr) \times A$ with functions that are $1$-periodic w.r.t. $y_1$.
  \end{remark}

\begin{remark}\label{rem:1}
 Our setting includes the simpler case in which  $f_{1,\per}$ and $\ell_{1,\per}$  do not depend on $y_1$.
\end{remark}

\begin{remark}\label{rem:0}
  In the present setting, the region of the state space in which the Hamiltonian $H_\varepsilon$ has fast variations is contained in  $\Omega_\varepsilon$.
\end{remark}

Let $\alpha$ be a positive discount factor.
 It is well known, see e.g. \cite{MR1484411}, that the value function $u_\varepsilon$ of the optimal control problem:
\begin{equation}\label{eq:12}
  u_\varepsilon(x)= \inf  \int_0 ^\infty    e^{-\alpha t} \ell_\varepsilon \left( z(t), a(t)\right) dt \quad  \hbox{subject to } \left\{
    \begin{array}[c]{l}
      \ds        z(t)= x+\int_0 ^t f_\varepsilon\left (z(\tau) ,a(\tau)\right) d\tau,\\
        a\in L^\infty(\R_+)\\
        a(t)\in A, \; \hbox{for almost } t\ge 0
    \end{array}
      \right.
\end{equation}
is the unique viscosity solution in $\rm{BUC}(\R^2)$ of 
\begin{equation}
  \label{eq:13}
\alpha \, u_\varepsilon+ H_\varepsilon\left( x, D u_\varepsilon\right)=0\quad \hbox{ in }\R^2.
\end{equation}
Our goal is to study the asymptotic behaviour of $u_\varepsilon$ as $\varepsilon\to 0$. 

\subsection{Main result}
\label{sec:main_result}

\begin{definition}[A flat stratification of $\R^2$] \label{def:1}
  The following partition  of $\R^2 $ arises naturally when we let $\varepsilon$  vanish:
  \begin{eqnarray}
    \label{eq:14}
    \R^2=\cM_2\cup\cM_1\cup\cM_0,\\
    \label{eq:15}    \cM_{0}=\{0\},\quad   \cM_{1}=(-\infty,0)\times \{0\},\quad
    \cM_{2}=\R^ 2\setminus \left(  \cM_{0} \cup  \cM_{1} \right).
  \end{eqnarray}
The submanifolds $ \cM_{0}$,  $\cM_{1}$ and $\cM_2$
are disjoint and 
  are respectively of dimension $0$, $1$ and $2$.  Clearly $\cM_2$ is an open subset of $\R^2$. In fact,  $ \cM_{0}$,  $\cM_{1}$ and $\cM_2$ form  an admissible  flat Whitney stratification of $\R^2$, see for example \cite[Definition 2.3]{MR4704058}.
\end{definition}
Our main result is the following:
\begin{theorem}
  \label{sec:main-result-1}
  We consider the solution $u_\varepsilon$ of (\ref{eq:13}), in the setting of  Subsection \ref{sec:setting}.
  As $\varepsilon\to 0$, the family $u_\varepsilon$ converges locally uniformly to
  a bounded and Lipschitz continuous function $u$ defined on $\R^2$, which is
  the unique solution to the following stratified problem
  associated to the admissible flat stratification of $\R^2$ defined in \eqref{eq:15}:
  \begin{enumerate}
  \item $u$ is a viscosity solution of
    \begin{equation}
      \label{eq:16}
      \alpha \, u+\overline H(\cdot,Du)= 0\quad  \hbox{in }\cM_2.
    \end{equation}
  \item
    \begin{enumerate}
    \item If  $\phi\in C^1(\R^2)$ is such that $u-\phi$ has a local minimum at $x\in \cM_1$, then
      \begin{equation}\label{eq:17}
        \alpha \, u(x)+\max\left( \HoneT(x, \partial_{x_1} \phi(x)),  \overline H(x,D\phi(x))\right)\ge 0.
      \end{equation}
      where $\HoneT: \overline{\cM_1}\times \R\ \to \R$ is the effective tangential Hamiltonian defined in Section \ref{sec:proof_2} below.
    \item  If  $\phi\in C^1(\cM_1)$ is such that $u-\phi$ has a local maximum at $x\in \cM_1$, then
      \begin{equation} \label{eq:18}
        \alpha \, u(x)+ \HoneT(x, \partial_{x_1} \phi(x)) \le 0.
      \end{equation}
    \end{enumerate}
  \item
    \begin{enumerate}
    \item
      If  $\phi\in C^1(\R^2)$ is such that $u-\phi$ has a local minimum at $0$, then
      \begin{equation}\label{eq:19}
        \alpha \, u(0)+\max\left( E,\HoneT(0, \partial_{x_1} \phi(0)),  \overline H(0,D\phi(0))\right)\ge 0,
      \end{equation}
      where $E$ is the  effective Dirichlet datum  defined in Section~\ref{sec:ergod-const-glob} below.
    \item
      \begin{equation}\label{eq:20}
        \alpha \, u(0)+E \le  0.
      \end{equation}
    \end{enumerate}
  \end{enumerate}
\end{theorem}

Note that Theorem  \ref{sec:main-result-1} is reminiscent of  \cite[Th. 1.1] {MR4643677} that deals with the homogenization of  a class of stationary Hamilton-Jacobi equations in  which the  Hamiltonian is obtained by perturbing  near the origin an otherwise periodic Hamiltonian. In \cite[Th. 1.1]{MR4643677}, the limiting problem consists of a Hamilton-Jacobi equation outside the origin, with the same effective Hamiltonian as in periodic homogenization, supplemented at the origin with an effective  Dirichlet condition  that keeps track of the perturbation. In the present case, the limiting problem presents singularities both on $\cM_1$ and at the origin, because before homogenization, the fast variations of the Hamiltonian are localized around $\overline \cM_1$.  The origin plays a particular role as the endpoint of $\cM_1$. Theorem  \ref{sec:main-result-1} can be seen as recursive version of  \cite[Th. 1.1] {MR4643677}, in terms of the dimensions of the involved submanifolds.

\begin{definition}\label{def:2}
  An upper semi-continuous function $u:\R^2\to \R$  which is a viscosity subsolution of \eqref{eq:16} in $\cM_2$ and satisfies  \eqref{eq:18} and  \eqref{eq:20} is named a weak viscosity subsolution of the stratified problem or {\sl weak stratified subsolution}, see  \cite[Definition 19.1]{MR4704058}.

  A lower semi-continuous function $u:\R^2\to \R$  which is a viscosity supersolution of \eqref{eq:16} in $\cM_2$ and satisfies  \eqref{eq:17} and  \eqref{eq:19} is named a supersolution of the stratified problem. In other words, $u$ is a supersolution in the sense of Ishii of the Hamilton-Jacobi equation whose Hamiltonian is discontinuous and coincides with $(x,p)\mapsto \overline H (x,p)$ at $x\in \cM_2$, with $(x,p)\mapsto \HoneT(x,p_1)$ at $x\in \cM_1$ and with $E$ at $x=0$,  see  \cite[Definition 2.1]{MR4704058}.

  The function $u$ arising in Theorem \ref{sec:main-result-1} is named a {\sl weak stratified solution} because it is both a weak stratified subsolution and
   supersolution of the stratified problem, see  \cite[Definition 19.1]{MR4704058}.
 \end{definition}

 \begin{remark} \label{rem:2}
  Consider a  weak subsolution $u$ of the stratified problem that is Lipschitz continuous in $\R^2$.
Since $u$ is a viscosity subsolution of $\alpha u+ \overline H (\cdot,Du )\le 0$ in $\cM_2$, it satisfies $\alpha  u(x)+ \overline H (x,D u(x) )\le 0$ at almost every $x\in \R^2$, see \cite[Prop. 1.9, Chapter I, page 31, and its proof]{MR1484411}. 
But $\overline H$ is continuous in $\R^2\times \R^2$ and convex w.r.t. its second argument. Therefore, from \cite[Prop. 5.1, Chapter II, page 77]{MR1484411}, $u$ is a viscosity subsolution of $\alpha v+ \overline H (\cdot,Dv )\le 0$ in the whole space $\R^2$.

As a consequence, the function $u$ arising in Theorem \ref{sec:main-result-1} is  a viscosity subsolution of $\alpha v+ \overline H (\cdot,Dv )\le 0$ in the whole space $\R^2$.
\end{remark}

\section{Proof of Theorem~\ref{sec:main-result-1}}
\label{sec:proof}
\subsection{Known facts}
\label{sec:proof_1}
Recall that our goal is to understand the asymptotic behaviour of $u_\varepsilon$ as $\varepsilon$ vanishes. 
First, using  either comparison principles, see for example \cite[Chapter II, Theorem 3.5]{MR1484411} or arguments from the theory of optimal control, we see  that
\begin{displaymath}
-M_\ell\le -\max_{y\in \R^2} H_\varepsilon(y,0) \le \alpha \, u_\varepsilon(x) \le -\min_{y\in \R^2} H_\varepsilon(y,0) \le M_\ell.  
\end{displaymath}
 From this estimate and \eqref{eq:13}, we infer from the coercivity of the Hamiltonian  that $u_\varepsilon$ is Lipschitz continuous in $\R^2$ with a Lipschitz constant independent of $\varepsilon$.

 In order to study the asymptotic behaviour of $u_\varepsilon$, we consider 
 \begin{eqnarray}
   \label{eq:21}
 \overline{u}(x)=\limsup_{\varepsilon\to 0}u_\varepsilon(x),
\\
\label{eq:22}
 \underline{u}(x)=\liminf_{\varepsilon\to 0}u_\varepsilon(x).
 \end{eqnarray}
 Note that, from the observation above on the regularity of $u_\varepsilon$,
 $ \overline{u}$ and  $\underline{u}$ coincide respectively with the half-relaxed semi-limits  $\ds \limsup_{x'\to x, \varepsilon\to 0}u_\varepsilon(x')$ and $\ds \liminf_{x'\to x, \varepsilon\to 0}u_\varepsilon(x')$, that are classically  used in the homogenization of Hamilton-Jacobi equations. It is clear that the functions $ \overline{u}$ and $ \underline{u}$  are  bounded and Lipschitz continuous.
 The following fact is well known:
\begin{lemma}
\label{sec:homog-refeq:7-1}
The functions $\overline{u}$ and $\underline{u}$ are respectively a bounded subsolution and a bounded supersolution of (\ref{eq:16}) in $\cM_2$.
\end{lemma}
\begin{remark} \label{rem:3}
  Arguing as in Remark \ref{rem:2}, we see that $\overline{u}$ is a viscosity subsolution of  $\alpha v+ \overline H (\cdot,Dv )\le 0$ in the whole space $\R^2$.
\end{remark}

\subsection{The strategy of proof  for Theorem~\ref{sec:main-result-1}}
The proof of Theorem~\ref{sec:main-result-1}  is done in five different steps that we now describe.
\begin{enumerate}
\item The first step (Subsection \ref{sec:proof_2}) is devoted
to obtaining the viscosity inequalities satisfied by
$\underline u$ and $\overline u$ at $x\in \cM_1$. Since $\cM_1$ is of codimension one, we may rely on the existing literature on homogenization leading to transmission boundary conditions on submanifolds of codimension one, see \cite{MR3299352, MR3565416,MR3912640, MR3441209, MR4450245}. For $p_1\in \R$ and $x\in \cM_1$, the construction of the effective tangential Hamiltonian $\HoneT(x,p_1)$ and of related correctors named $\xi(x, p_1,\cdot)$ is a key difficulty in the theory. The latter correctors  are  defined in unbounded domains because there is no periodicity w.r.t. $y_2$, and they are constructed by introducing a sequence of ergodic problems in truncated cells. The effective tangential Hamiltonian keeps  memory of the fast variations of the coefficients near $\cM_1$. From the control theoretic viewpoint, it accounts for the trajectories that remain close to $\cM_1$.
Since the proofs are similar to those contained in the above-mentioned articles, we will omit them.
  
\item The second step (Subsection \ref{sec:ergod-const-glob}) consists of constructing the ergodic constant $E$ associated to the origin and a related  corrector $w$. This part is reminiscent of the arguments of \cite{MR4643677}, in which the authors studied the homogenization of a periodic Hamilton-Jacobi equation with a defect of periodicity located near the origin.  As above, an important difficulty is that the corrector $w$  must be a function defined in the whole space $\R^2$, which makes it necessary to impose some  condition at infinity.  We will see that the latter amounts to the fact that  $w$ is the locally uniform limit as $R\to +\infty$ of a family $(w^R)_{R>0}$ of solutions of
problems with state constraints posed in the balls $B_R(0)$.  
From the optimal control theory viewpoint,  these problems,  referred to as {\sl truncated cell problems}, account for trajectories that 
remain close to the origin. 
\item In the third step (Subsection \ref{sec:function--overline}), we prove  that  the upper-limit $\overline u$ satisfies  (\ref{eq:20}), relying on
Evans' method of perturbed test-functions, see \cite{MR1007533}.  The construction of the perturbed test-function involves the above mentioned solution $w^R$ to the  truncated cell problem in the ball $B_R(0)$.
\item The fourth step (Subsection \ref{sec:function--underline}) consists of proving that the lower-limit $\underline u$ satisfies condition (\ref{eq:19}). The main idea is to construct subcorrectors by combining in a suitable way the function $w$ mentioned above and the correctors $\xi(0,p,\cdot)$ associated to the Hamiltonian $H_{1,\per}(0,\cdot)$. We take benefit of the present setting to revisit  the  proof presented in  \cite{MR4643677} and simplify it by using subcorrectors instead of correctors.
\item The last  step of the proof (Subsection \ref{proof_of_main_th}) mostly consists of deducing from the previously obtained results that $\underline u=\overline u$, by means of a comparison principle. Theorem~\ref{sec:main-result-1} then follows.
\end{enumerate}

\subsection{The effective Hamiltonian $\HoneT$ and the viscosity inequalities satisfied by  $\overline{u}$ and $\underline{u}$ on $\cM_1$}
\label{sec:proof_2}

In this subsection, we are going to discuss the effective problem arising on $\cM_1$ as $\varepsilon\to 0$. Because the submanifold  $\cM_1$ is of codimension one, we can rely on  results obtained in the last decade and concerning the analysis of Hamilton-Jacobi equations posed on heterogeneous structures such as networks \cite{MR3621434,MR3358634, MR3556345}, {\sl booklet}-like geometries or multidimensional junctions \cite{barles2013bellman,MR3690310, oudet2015equations}. These problems all involve Hamilton-Jacobi equations with nonlinear transmission conditions on submanifolds with codimension one. In particular, the works \cite{MR3565416} and \cite{MR3912640} were concerned with  Hamilton-Jacobi equations  in an environment consisting of two different homogeneous media separated by an oscillatory interface. The oscillations of the interface have small period and amplitude, and as the latter parameter vanishes, the interface tends to an hyperplane. At the limit when both parameters vanish, one finds on the flat interface an effective nonlinear  transmission condition  keeping memory of the previously mentioned microscopic oscillations. Similarly, \cite{MR3441209} is devoted to a family of time dependent Hamilton-Jacobi equations  on the simplest possible network composed of two half-lines with a perturbation of the Hamiltonian localized in a small region close to the junction.  Related homogenization problems with applications  to traffic flows were discussed in \cite{MR3809148,MR4450245}.
Key arguments in \cite{MR3299352, MR3565416, MR3912640, MR3441209} are the construction of families of correctors  that account for the localized perturbations of the environment and are defined in unbounded domains. The same strategy will be used here in order to obtain the effective Hamiltonian $\HoneT$.

Although  \cite{MR3565416} and \cite{MR3912640}  are focused on the homogenization of periodic interface problems while in the present case, the Hamiltonian depends continuously on the state variable,
the proofs of the results  contained in the present subsection are quite similar to those contained
in the latter references.  They  will be omitted.

\subsubsection{Construction of the
  effective Hamiltonian $\HoneT$ via truncated cell problems}
\label{effective_H_M1}
 Since we have defined $H_{1,\per}(x,\cdot,\cdot)$ at all $x\in \R^2$, we are going to construct $\HoneT(x,\cdot)$ at all $x\in \R^2$, although  for what follows,
 it would be sufficient to consider only  $x\in \overline{\cM_1}$.
 
Fix $x\in \R^2$ and $p_1\in \R$. We define  the unbounded  cell $Y=(\R/\Z) \times \R$, and, given a parameter $\rho>R_0$ that will eventually tend to $+\infty$,  the truncated cell $Y_\rho=(\R/\Z) \times (-\rho, \rho)$. The {\sl truncated cell problem} associated with $x\in \R^2$ and $p_1\in \R $ is as follows: 
\begin{eqnarray}
\label{trunc-cellp1}
    H_{1,\per}(x,y, p_1e_1+ D\xi_\rho(x,p_1,y)) &\le \lambda_\rho(x, p_1)&\hbox{ if } y\in Y_\rho,\\
   \label{trunc-cellp2} H_{1,\per}(x,y, p_1e_1+ D\xi_\rho(x,p_1,y)) &\ge \lambda_\rho( x,p_1)&\hbox{ if } y\in \overline{Y_\rho},
\end{eqnarray}
where the inequalities are understood in the sense of viscosity.
\begin{remark}\label{rem:4}
  Problem \eqref{trunc-cellp1}-\eqref{trunc-cellp2} can be seen as an ergodic problem in  $Y_\rho$ associated with state constraints on the boundaries $\{y: y_2=\pm\rho\}$.
 An equivalent formulation may be written by replacing $Y_\rho$ by $(0,1)\times (-\rho,\rho)$ and by additionally imposing  periodicity with period $1$ in the variable $y_1$.
\end{remark}

\begin{lemma}
\label{lem:existence_solution_truncated_cell_pb}
There is a unique  $\lambda_\rho(x, p_1)\in \R$ such that \eqref{trunc-cellp1}-\eqref{trunc-cellp2} admits a viscosity solution. 
For this choice of  $\lambda_\rho( x,p_1)$, there exists a  solution  $\xi_\rho(x,p_1,\cdot)$
 that is Lipschitz continuous with  Lipschitz constant
 $L$ depending on $p_1$ only (uniform in $x$ and $\rho$).
\end{lemma}

As in \cite{MR3299352,MR3565416},  using the optimal control interpretation of \eqref{trunc-cellp1}-\eqref{trunc-cellp2}, it is easy to prove that for 
a positive $K$ that may depend on  $p_1$ but not on $x$ and $\rho$, and for 
all $R_0<\rho_1\le \rho_2$, 
\[\lambda_{\rho_1}( x,p_1)\leq \lambda_{\rho_2}(x, p_1)\leq K.\]
The effective tangential Hamiltonian $\HoneT(x,p_1)$ is defined by 
\begin{equation}\label{eq_Hone}
 \HoneT(x,p_1) =\lim_{\rho\rightarrow \infty} \lambda_{\rho}( x,p_1),\quad  \hbox{ for all } p_1\in \R.
\end{equation}
Important properties of  $\HoneT$
will be needed for obtaining comparison principle relative to  
 the effective problem in Theorem \ref{sec:main-result-1}.
 These properties are inherited from the original Hamiltonian,
 as established in the pioneering work \cite{LPV}.
\begin{lemma}
\label{lem:regularity_of_ergodic_hamiltonian}
For any $x\in \R^2$, the function $p_1\mapsto \HoneT(x,p_1)$ is convex.
There exist positive constants $L_1, C_1, c_1, m_1$ and a modulus of continuity $\omega_1$ such that,
for any $x,\tilde x \in \R^2$     $p_1, \tilde p_1\in \R$,
 \begin{eqnarray}
   \label{lem:E_lipschitz_wrt_p_2}
\mid \HoneT(x,p_1)-\HoneT(x,\tilde p_1) \mid \le L_1 |p_1-\tilde p_1|,
\\
\label{lem:E_lipschitz_wrt_z_2}
   \mid  \HoneT(x,p_1)- \HoneT(\tilde x,p_1) \mid \le C_1|p_1| |x-\tilde x|+
   \omega_1(|x-\tilde x |),
\\
 \label{lem:E_coercive}
 \frac 1 {c_1} |p_1|-m_1\le \HoneT(x,p_1)\le c_1 |p_1|+m_1.
\end{eqnarray}
\end{lemma}

\subsubsection{Correctors}
\label{correctors_and_cell_pb}
The following theorem is proved  in the same way as Theorem 4.8 in \cite{MR3565416}.
\begin{theorem}
\label{thm:stability_from_truncated_cell_pb_to_global_cell_pb}
Let $\xi_\rho( x,p_1,\cdot)$ be a sequence of uniformly Lipschitz continuous solutions of the truncated cell-problem \eqref{trunc-cellp1}-\eqref{trunc-cellp2} that converges to $ \xi( x,p_1,\cdot)$
locally uniformly in $Y$. Then  $ \xi(x,p_1,\cdot)$ is a Lipschitz continuous viscosity solution of the following equation posed in $Y$:
\begin{equation}
   \label{cellpE} H_{1,\per}(x,y, p_1e_1+ D \xi(x,p_1,y)) = \HoneT(x, p_1),
 \end{equation}
which, we recall, means that $\xi$ is $1$-periodic with respect to $y_1$ and satisfies 
 \eqref{cellpE} in the viscosity sense at all $y\in \R^2$.

\end{theorem}

 By subtracting  $\xi(x,p_1, 0)$ to $\xi_\rho(x,p_1,\cdot)$ and $\xi(x,p_1,\cdot)$, we may  assume that
 \begin{equation}
    \label{cellpE2}
\xi(x,p_1,0)=0.   
\end{equation}
Hereafter, we will always consider that $\xi(x,p_1,\cdot)$ satisfies  \eqref{cellpE} and 
 \eqref{cellpE2}.

For $\varepsilon>0$,  let us set $\Xi_\varepsilon(x,p_1,y)=\varepsilon\xi( x,p_1,\frac {y} {\varepsilon})$. The following result can be proved in the same way as \cite[Theorem 4.6,iii]{MR3441209}:
 \begin{lemma}
\label{lem:rescaling_omega} 
 For any $(x,p_1)\in \R^2\times \R$,   
there exists a sequence    $\varepsilon_n$ of positive numbers tending to $0$ as $n\to +\infty$  such that  $\Xi_{\varepsilon_n}  (x,p_1,\cdot)$ converges locally uniformly to a 
Lipschitz  function $y\mapsto \Xi( x,p_1, y)$ (the Lipschitz constant does not depend on $x$). 
This function is constant with respect to $y_1$ and satisfies $\Xi( x,p_1,0)=0$.  Finally, $y\mapsto \Xi(x,p_1,y)$   is  a viscosity solution  of 
\begin{equation}
\label{W}
\overline H\left(x,\frac {d\Xi }{dy_2}((x,p_1,y) e_2+p_1e_1\right)=\HoneT(x,p_1)
\end{equation}
at all $y\in \R^2$ such that  $y_2\not=0$.
\end{lemma}

As a consequence of Lemma~\ref{lem:rescaling_omega}, 
 $\HoneT(x,p_1)$ is bounded from below by a quantity depending on $\overline H(x,\cdot)$:
 \begin{corollary}
\label{cor:E_bigger_than_E_0}
For any  $(x,p_1)\in \R^2 \times \R$,
\begin{equation}
\label{eq:cor:E_bigger_than_E_0}
\HoneT(x,p_1) \geq  \min_{q\in \R}\overline H(x,p_1e_1+qe_2).
\end{equation}
\end{corollary}
Of course, \eqref{eq:cor:E_bigger_than_E_0} implies that $\HoneT(x,p_1) \geq  \min_{q\in \R^2}  \overline H(x,q)\ge -M_\ell$.
\begin{definition}\label{def:3} $\;$\\
  For   $(x,p)\in \R^2\times \R^2$, we set
  \begin{eqnarray}
    \label{eq:23}
    \overline{H}_{\downarrow}(x,p)= \max_{a\in A:  \bar{f}(x,a)\cdot e_2 \ge 0  } \Bigl(-p\cdot \bar{f}(x,a)-\bar{\ell}(x,a)\Bigr) ,\\
     \label{eq:24}
    \overline{H}_{\uparrow}(x,p)= \max_{a\in A:  \bar{f}(x,a)\cdot e_2 \le 0  } \Bigl(-p\cdot \bar{f}(x,a)-\bar{\ell}(x,a)\Bigr) .
  \end{eqnarray}
  The function $p_2\mapsto  \overline{H}_{\downarrow}(x,p_1 e_1+p_2e_2)$, resp.
 $p_2\mapsto  \overline{H}_{\uparrow}(x,p_1 e_1+p_2e_2)$
 is the nonincreasing, resp. nondecreasing envelope of $p_2\mapsto  \overline{H}(x,p_1 e_1+p_2e_2)$. The following identities hold:
 \begin{eqnarray}
   \label{eq:25}
   \overline{H}(x,p)&=& \overline{H}_{\downarrow}(x,p)+ \overline{H}_{\uparrow}(x,p)- \min_{q\in \R}\overline H(x,p_1e_1+qe_2)\\
   \notag &=& \max\left(   \overline{H}_{\downarrow}(x,p), \overline{H}_{\uparrow}(x,p) \right).
 \end{eqnarray}

From Corollary \ref{cor:E_bigger_than_E_0} and the coercivity
of  $\overline H$,  the following quantities are well defined for all
 $(x,p_1)\in \R^2\times \R$,
\begin{eqnarray}
\label{eq:26}
  \overline{\Pi}(x,p_1)= \max\left\lbrace q\in \R : 
  \overline H_\uparrow( x,p_1e_1+qe_2)=\HoneT(x,p_1) \right\rbrace,\\
  \label{eq:27}
  \underline{\Pi}(x,p_1)= \min\left\lbrace q\in \R : 
  \overline H_\downarrow( x,p_1e_1+qe_2)=\HoneT(x,p_1) \right\rbrace.
\end{eqnarray}
\end{definition}
  
The two propositions that follow can be proved in the same way as  \cite[Prop. 4.13 and 4.14]{MR3565416}:
\begin{proposition}
  \label{cor:slopes_omega}
For   $(x,p_1)\in \R^2\times \R$,
if $\HoneT(x, p_1)>\min_{p_2\in \R} \overline H(x,p_1e_1+p_2 e_2)$, then  there exist  $\rho^*=\rho^*(x,p_1)  >0$ and $M^*= M^*(x,p_1)   \in \R$
such that
\begin{enumerate}
\item for all $(y_1,y_2)\in (\R/\Z)\times  [\rho^*,+\infty)$, $h_2\ge 0$ and $h_1\in \R$,
\begin{equation}
\label{slope32}
\xi(x, p_1, y+h_1e_1+h_2 e_2)-\xi(x, p_1, y)\geq
\overline{\Pi}(x,p_1) h_2-M^*.
\end{equation}
\item  for all  $(y_1,y_2)\in (\R/\Z)\times (-\infty,-\rho^*]$, $h_2\ge 0$ and $h_1\in \R$,
\begin{equation}
  \label{slope3_bis2}
  \xi(x, p_1, y+h_1e_1-h_2 e_2)-\xi(x, p_1, y)\geq
  -\underline{\Pi}(x,p_1) h_2-M^*.
  \end{equation}
\end{enumerate}
\end{proposition}

\begin{proposition}
\label{cor:control_slopes_W}
Consider   $(x,p_1)\in \R^2\times \R$.
\\
If $\HoneT(x,p_1) >  \min_{q\in \R}\overline H(x,p_1e_1+qe_2)$,
  then
  \begin{equation}
     \label{eq:28}
     \Xi(x,p_1,y)=   \overline{\Pi}(x,p_1) y_{2,+} - \underline{\Pi}(x,p_1) y_{2,-}.
   \end{equation}
If $\HoneT(x,p_1)= \min_{q\in \R}\overline H(x,p_1e_1+qe_2)$,
  then
\begin{eqnarray}
\label{cor:control_slopes_W1}
  \underline{\Pi}(x,p_1)  \le \partial_{y_2}\Xi(x,p_1,y) \le \overline{\Pi}(x,p_1), \\
\label{eq:control_slopes_W_summary}
 \underline{\Pi}(x,p_1) y_{2,+} - \overline{\Pi}(x,p_1) y_{2,-}
  \le \Xi(x,p_1,y) \le    \overline{\Pi}(x,p_1) y_{2,+} - \underline{\Pi}(x,p_1) y_{2,-}.
\end{eqnarray}
\end{proposition}

\subsubsection{Viscosity inequalities satisfied by $\overline u$ and $\underline u$ on $\cM_1$}
The following theorem is proved by using Evans' method of perturbed test-function, see \cite{MR1007533}, involving the functions $\Xi(x,p_1,\cdot)$ and the related correctors $\xi(x,p_1,\cdot)$ for suitable vectors $p$. Its proof follows the same lines as that of  \cite[Th. 1.5]{MR3565416}.
\begin{theorem}\label{main_theorem_secM1}
  Consider $x\in \cM_1$.

  If $\phi: \R^2 \to \R$ is a continuous fonction  with $\phi|_{\R\times [0,+\infty)} \in C^1( \R\times [0,+\infty) )$ and
  $\phi|_{\R\times (-\infty,0]} \in C^1( \R\times (-\infty,0] )$ and such that $x$ is a  local maximum of $\overline u -\phi$, then
  \begin{equation}
    \label{eq:29}
    \begin{split}
&    \alpha \overline u(x)+ \max\left( \HoneT(x, \partial_{x_1} \phi(x)), \overline H_{\downarrow} \left(x, D \phi|_{\R\times [0,+\infty)}(x)\right),
      \overline H_{\uparrow} \left(x, D \phi|_{\R\times (-\infty,0]}(x)\right)\right) \\ &\le 0.      
    \end{split}
  \end{equation}

  If $\phi: \R^2 \to \R$ is a continuous fonction with $\phi|_{\R\times [0,+\infty)} \in C^1( \R\times [0,+\infty) )$ and  $\phi|_{\R\times (-\infty,0]} \in C^1( \R\times (-\infty,0] )$ and such that $x$ is a  local minimum of $\underline u -\phi$, then
  \begin{equation}
    \label{eq:30}
     \begin{split}
  &  \alpha \underline u(x)+ \max\left( \HoneT(x, \partial_{x_1} \phi(x)), \overline H_{\downarrow} \left(x, D \phi|_{\R\times [0,+\infty)}(x)\right),
      \overline H_{\uparrow} \left(x, D \phi|_{\R\times (-\infty,0]}(x)\right)\right) \\ &\ge 0.
  \end{split}
\end{equation}
\end{theorem}

We conclude Subsection \ref{sec:proof_2} by the following corollary of Theorem \ref{main_theorem_secM1}, which states that  $\overline u$ and  $ \underline u$ are respectively
sub and supersolutions of  \eqref{eq:18} and  \eqref{eq:17}. This will yield the effective problem on $\cM_1$, i.e. item 2. in Theorem  \ref{sec:main-result-1}.
\begin{corollary} \label{corollary:2}
  Consider $x\in \cM_1$.

  If $\phi \in C^1 (\cM_1)$ is such that $x$ is a local maximum of $\overline u -\phi$ on $\cM_1$, then
  \begin{equation}
    \label{eq:31}
   \alpha \overline u(x)+  \HoneT(x, \partial_{x_1} \phi(x)) \le 0.
 \end{equation}

  If $\phi \in C^1 (\R^2)$ is such that $x$ is a local minimum of $\underline u -\phi$, then
  \begin{equation}
    \label{eq:32}
   \alpha \underline u(x)+ \max\left( \HoneT(x, \partial_{x_1} \phi(x)), \overline H (x, D\phi(x))\right)  \ge 0.
  \end{equation}
\end{corollary}

\begin{proof}
Obtaining \eqref{eq:32} from  \eqref{eq:30} is straightforward.

Let $\phi \in C^1 (\cM_1)$ be such that $x$ is a local maximum of $\overline u -\phi$ on $\cM_1$.
  It is possible to extend $\phi$ outside $\cM_1$ by a continuous function $\widetilde \phi: \R^2\to \R$ of the following form:
  \begin{displaymath}
    \widetilde \phi (z)= \phi(z_1) + B |z_2|,
  \end{displaymath}
  where the positive real number $B$ will be chosen soon.
Clearly,  $\widetilde \phi|_{\{z: z_2\ge 0\}} \in C^1( \R\times [0,+\infty) )$ and  $\widetilde \phi|_{\{z: z_2 \le 0\}} \in C^1( \R\times (-\infty,0] )$.
  We choose $B$ in order to satisfy two conditions:
  \begin{enumerate}
  \item $B$ is greater than the Lipschitz constant of $\overline u$. This implies that $x$ is a  local maximum of $\overline u -\widetilde \phi$.
  \item $B$ is sufficiently large  such that
    \begin{eqnarray*}
      \overline H_{\downarrow} \left(x, D \widetilde \phi|_{\R\times [0,+\infty)}(x)\right) =  \overline H_{\downarrow} \left(x, D\phi(x)+B e_2\right)= \min_{q_2}  \overline H \left(x, D\phi(x)+q_2e_2\right),\\
     \overline H_{\uparrow} \left(x, D \widetilde \phi|_{\R\times (-\infty,0]}(x)\right)  = \overline H_{\uparrow} \left(x, D \phi(x)-Be_2\right)= \min_{q_2}  \overline H \left(x, D\phi(x)+q_2e_2\right).
    \end{eqnarray*}
  \end{enumerate}
  It is clear that $B$ satisfying the condition above exists.
Thanks to \eqref{eq:cor:E_bigger_than_E_0},
 \begin{displaymath}
    \HoneT(x, \partial_{x_1} \phi(x))\ge   \max\left(   \overline H_{\downarrow} \left(x, D \widetilde \phi|_{\R\times [0,+\infty)}(x)\right),
      \overline H_{\uparrow} \left(x, D \widetilde \phi|_{\R\times (-\infty,0]}(x)\right) \right).
  \end{displaymath}
  %
  Because \eqref{eq:29} is satisfied by $\widetilde \phi$,  the observation above yields  \eqref{eq:31}.  
\end{proof}
We have found the inequality satisfied by $ \underline u$ and  $\overline u$ at $x\in \cM_1$. We can now concentrate on the effective equations at the origin.
\subsection{The ergodic constant  $E$ and related correctors}
\label{sec:ergod-const-glob}
The construction that follows is identical to that introduced in \cite{MR4643677}, and we repeat it for consistency and  convenience of the reader.
In order to understand  the asymptotics of $u_\varepsilon$ near the origin, we start by solving {\sl truncated cell problems} in balls, associated to  state constrained boundary conditions. From the optimal control theory viewpoint,  these problems account for trajectories that  remain close to the origin.\\
For $\lambda>0$, $R>0$, we know from e.g. \cite{MR838056,MR951880} that there exists a unique function $w^{\lambda, R} \in C(\overline{B_R(0)})$ such  that
\begin{eqnarray}
  \label{eq:33}
\lambda w^{\lambda, R} + H(0,y, Dw^{\lambda, R}) &\le& 0 \quad \hbox{ in } B_R(0),\\
\label{eq:34}
\lambda w^{\lambda, R} + H(0, y, Dw^{\lambda, R}) &\ge& 0 \quad \hbox{ in } \overline {B_R(0)},
\end{eqnarray}
the above inequalities being understood in the sense of viscosity.  The function $w^{\lambda, R}$ is the value function of the following infinite horizon state constrained optimal control problem in $\overline{B_R(0)}$,
\begin{eqnarray}\label{eq:35}
  w^{\lambda, R}(z)=\\ \notag  \inf  \int_0 ^\infty    e^{-\lambda t} \ell \left(0, y(t), a(t)\right) dt \quad  \hbox{subj. to } \left\{
    \begin{array}[c]{l} \ds
        y(t)= z+\int_0 ^t f\left (0, y(\tau) ,a(\tau)\right) d\tau\\
        y(t)\in \overline{B_R(0)},\\
        a\in L^\infty(\R_+) \\ a(t)\in A, \; \hbox{for almost } t \ge 0.
    \end{array}
\right.
\end{eqnarray}
Since $\ell(0,\cdot,\cdot)$ is bounded on $\R^2 \times A$,
  $\lambda  \|w^{\lambda, R}\|_{L^\infty(B_R(0))}$ is bounded uniformly in $\lambda$ and $R$. More precisely,
 $  \min_{(y,a)\in \R^2\times A} \ell(0,y,a) \le   \lambda  w^{\lambda, R} \le \max_{(y,a)\in \R^2\times A} \ell(0,y,a)$.
This and  the uniform coercivity of $H$ imply with \eqref{eq:33} that $\|Dw^{\lambda,R}\|_ {L^\infty(B_R(0))}  $ is bounded uniformly in $\lambda$ and $R$.

Using Ascoli-Arzel{\`a} theorem, we may suppose that up to the extraction of a sequence,  as $\lambda\to 0$,  $\lambda w^{\lambda, R}$ 
tends uniformly on $\overline {B_R(0)}$ to some {\sl ergodic} constant $-E^R$ that is bounded from above and below uniformly in $R$,   and that $w^{\lambda, R} -w^{\lambda, R}(0)$ tends uniformly 
on $\overline {B_R(0)}$ to some function $w^R$ such that $w^R(0)=0$ and which is Lipschitz continuous with a Lipschitz constant independent of $R$.  By classical results on the stability of viscosity solutions 
of state constrained problems,  $w^R$ is a viscosity solution of 
\begin{eqnarray}
\label{eq:36}
 H(0,y, Dw^{ R}) &\le& E^R \quad \hbox{ in } B_R(0),\\
\label{eq:37}
 H(0,y, Dw^{R}) &\ge& E^R \quad \hbox{ in } \overline {B_R(0)}.
\end{eqnarray}
The comparison principle for state constrained problems, see \cite{MR838056,MR951880}, yields the uniqueness 
of $E^R$ such that (\ref{eq:36})-(\ref{eq:37}) has a solution. Thus, 
\begin{displaymath}
\ds \lim_{\lambda\to 0} \| \lambda w^{\lambda, R}+ E^R\| _{C(\overline B_R(0))}=0, 
\end{displaymath} i.e. uniform convergence and not only for a subsequence.

We deduce for example from (\ref{eq:35})  that 
\begin{displaymath}
  R_1\ge R_2 \quad \Rightarrow \quad \lambda w^{\lambda, R_1} \le \lambda w^{\lambda, R_2}, 
\end{displaymath}
and passing to the limit as $\lambda\to 0$, we obtain the monotonicity property of the ergodic constants $E^R$:
\begin{equation}
  \label{eq:38}
  R_1\ge R_2 \quad \Rightarrow \quad E^{ R_1} \ge E^{ R_2}. 
\end{equation}
Since $E^R$ is bounded from above independently of $R$,
\eqref{eq:38}  implies that
\begin{equation}
   \label{eq:39}
   E=\lim_{R\to \infty} E^R
\end{equation}
 exists in $\R$.

 Similarly, since  $w^R(0)=0$ and  $w^R$ is Lipschitz continuous on $\overline {B_R(0)}$ with a Lipschitz constant independent of $R$, we may construct by Ascoli-Arzel{\`a} theorem and  a diagonal extraction argument a sequence $(R_n)_{n\in \N}$,  $R_n\to +\infty$ as $n\to \infty$,  such that $w^{R_n}$ tends to some function $w$ locally uniformly in $\R^2$. We then see that $w(0)=0$ and $w$ is a  Lipschitz continuous viscosity solution of 
\begin{equation}
  \label{eq:40}
H(0,y, Dw)= E \quad \hbox{in }\R^2.
\end{equation}
Let us now zoom out and consider the function $w_\varepsilon: x \mapsto \varepsilon w(\frac x \varepsilon)$; it is clearly a viscosity solution of $ H(0,\frac x \varepsilon , D_x w_\varepsilon)= E$, and it is Lipschitz continuous with the same constant as $w$. Hence, after the extraction of a sequence, we may assume that 
$w_\varepsilon$ converges locally uniformly to some Lipschitz function  $W$ on $\R^2$. A standard argument yields  that $W$ is a viscosity solution of $\overline H (0,DW)=E$ in $\cM_2$.  This implies  
\begin{equation}
  \label{eq:41}
E\ge \min_{p\in \R^2 }  \overline H (0,p).
\end{equation}

Arguing as in Subsection \ref{sec:proof_2}, in particular Theorem \ref{main_theorem_secM1} and Corollary  \ref{corollary:2}, or as in
 \cite{MR3565416},  it can also be proved that
$W|_{\cM_1}$ is a viscosity subsolution of $\HoneT(0,\partial _{x_1} W)=E$ on $\cM_1$
and that if  $\phi\in C^1(\R^2)$ is such that $W-\phi$ has a local minimum at $x\in \cM_1$, then
\begin{displaymath}
  \max\left( \HoneT({\YA{0}}, \partial_{x_1} \phi(x)),  \overline H({\YA{0}},D\phi(x))\right)\ge 0.
\end{displaymath}
The fact that $W|_{\cM_1}$ is a viscosity subsolution of $\HoneT(0,\partial _{x_1} W)=E$ on $\cM_1$ implies
\begin{equation}
  \label{eq:42}
E\ge \min_{p_1\in \R }  \HoneT (0,p_1),
\end{equation}
which is stronger than \eqref{eq:41} because of \eqref{eq:cor:E_bigger_than_E_0}.

\subsection{Upper bound on  $ \overline  u(0)$}
\label{sec:function--overline}
\begin{proposition}
  \label{sec:ergod-const-glob-1}
The upper limit $\overline u$ is such that $\alpha \overline u(0)\le -E$.
\end{proposition}
\begin{proof}
   Let us proceed by contradiction and assume that 
  \begin{equation}
    \label{eq:43}
 \alpha \overline u(0) +E = \theta >0.
  \end{equation}
 Using $w^R$ defined in Subsection~\ref{sec:ergod-const-glob} (recall  $w^R (0)=0$), let us set
\begin{displaymath}
  \phi^{\varepsilon, R}= \overline u (0)+ \varepsilon w^R (\frac x \varepsilon).
\end{displaymath}
We deduce  from (\ref{eq:37}) and  \eqref{eq:43}  that $ \phi^{\varepsilon, R}$ is a viscosity supersolution of 
 \begin{displaymath}
  \alpha \phi^{\varepsilon, R}(x) + H\left(0,\frac x \varepsilon, D \phi^{\varepsilon, R}  \right)\ge \alpha \varepsilon w^R\left(\frac x \varepsilon\right)+E^R-E+\theta\quad \hbox{in  } \overline{B_{\varepsilon R}(0)}.
\end{displaymath}
There exists $r>0$
such that  $E^R - E\ge -\frac \theta 4$ for any $R\ge r$. Let us fix such a value of $R$. 
Having fixed $R$, we see that  for $\varepsilon_0$ sufficiently small and any $\varepsilon$ such that $0<\varepsilon\le \varepsilon_0$,
\begin{displaymath}
  \alpha \varepsilon w^{R}\left(y\right) \ge -\frac \theta 4 \quad \hbox{for any } y\in \overline{B_R(0)}.
\end{displaymath}
We deduce that, for any $\varepsilon\le \varepsilon_0$,
\begin{equation}\label{eq:44}
  \alpha \phi^{\varepsilon, R} + H\left(0,\frac x \varepsilon, D \phi^{\varepsilon, R}  \right)\ge
 \frac \theta 2 \quad \hbox{in  } \overline {B_{\varepsilon R}(0)}.
\end{equation}

Next, using \eqref{eq:21}, consider a vanishing sequence $0<\varepsilon_n\le \varepsilon_0$ such that  $u_ {\varepsilon_n}(0)$ tends to $\overline u (0)$.

We know that $u_ {\varepsilon_n}$ satisfies in the sense of viscosity
\begin{displaymath}
  \alpha \, u_ {\varepsilon_n} + H\left(x,\frac x {\varepsilon_n}, Du_ {\varepsilon_n} \right)\le
  0 \quad \hbox{in  } B_{\varepsilon_n R}(0).
\end{displaymath}
Recall that $u_ {\varepsilon_n}$  is Lipschitz continuous with a Lipschitz constant uniform in $\varepsilon_n$ and $R$. From this observation and from \eqref{eq:9}, we see that for $\varepsilon_n$ sufficiently small, $ u_ {\varepsilon_n}$  satisfies in the sense of viscosity
\begin{equation} \label{eq:45}
  \alpha \, u_ {\varepsilon_n} + H\left(0,\frac x {\varepsilon_n}, Du_ {\varepsilon_n} \right)\le
  \frac \theta  4 \quad \hbox{in  } B_{\varepsilon_n R}(0).
\end{equation}
From  \eqref{eq:44} and \eqref{eq:45}, the comparison principle for Hamilton-Jacobi equations with state constraints,  see \cite{MR838056,MR951880} and \cite[Th. 5.8, Chapter IV, page 278]{MR1484411}, then implies that
\begin{displaymath}
  \phi^{{\varepsilon_n}, R}- \frac  \theta {4 \alpha}\ge  u_ {\varepsilon_n}  \quad \hbox{in  } B_{\varepsilon_n R}(0),
\end{displaymath}
Taking $x=0$ and letting $n$ tend to $+\infty$ yields $  \overline u (0)- \frac  \theta {4 \alpha}\ge \overline u(0)$, the desired contradiction.
\end{proof}

\subsection{Study of $\underline u$ near the origin}
\label{sec:function--underline}

We first state and prove two technical lemmas, namely  Lemma \ref{lem:2} and  \ref{lem:3},
where the fixed vector $p\in \R^2$ is a placeholder for the gradient $D\phi$ of the test-function  at the origin  that will later be used (in Propositions  \ref{sec:function--underline-2} through   \ref{sec:function--underline-4})
in order to establish  the properties of item 3. within Theorem \ref{sec:main-result-1}.

\begin{lemma} \label{lem:2}
  Consider $p\in \R^2$. If $    \max(E,\HoneT(0,p_1))<\overline H (0,p)$, then there exists
  a Lipschitz function $\chi:\R^2\to \R$ such that
  \begin{eqnarray}
    \label{eq:46}
\chi(y)\le p\cdot y,\quad \hbox{for all } y\in \R^2,
    \\
    \label{eq:47}   H(0,y,D\chi(y))\le \overline H(0,p), \quad \hbox{in }\R^2,
  \end{eqnarray}
  where \eqref{eq:47} is understood in the viscosity sense.
\end{lemma}
\begin{proof}
 From the convexity of $\HoneT(0,\cdot)$, there exists a unique $q_1>p_1$ such that $\HoneT(0,q_1)=\overline H(0,p)$.
  Because $q_1>p_1$,
  there  exists a constant $c>0$ such that $q_1y_1+\xi(0,q_1,y)-c< p\cdot y$ for all $y\in\overline \Omega$, recalling that $\xi(0,q_1,\cdot)$ is defined in paragraph \ref{correctors_and_cell_pb}.
 For such a constant $c$, there exists $R_1>R_0$ (recall that $R_0$ is fixed in paragraph \ref{sec:setting}, see   \eqref{eq:1}) such that
  \begin{equation}\label{eq:48}
    q_1y_1+\xi(0,q_1,y) -c > p\cdot y,\quad \hbox{for all } y\in D,
  \end{equation}
  where
  \begin{equation}
    \label{eq:49}
    D= \{y\in \R^2\;:\;  y_1>0, \,  |y| \ge R_1  \, |y_2|\le R_0 \}
  \end{equation}
Note that $\overline H (0,  q_1e_1+D\xi(0,q_1,y) ) = \overline H(0,p)$ in $|y_2|>R_0$. 
\\
With $w$ defined in Subsection \ref{sec:ergod-const-glob},
   we may then choose $C>0$ such that $w(y)-C <  \min(q_1y_1+\xi(0,q_1,y)-c, p\cdot y)  $  for all $y\in \overline { B_{R_1}(0)}$.
   We deduce that the function $\chi: y\mapsto \min(  w(y)-C, q_1y_1+\xi(0,q_1,y)-c, p\cdot y) $ is Lipschitz continuous  as the minimum of three Lipschitz continuous functions.
To prove that $\chi$ is also an almost everywhere subsolution of   \eqref{eq:47},  we split the space~${\mathbb R}^2$ into four regions (the first two of that overlap each other):
\begin{description}
\item[(i)] for $y\in \overline{B_{R_1}(0)}$, the constant~$C$ has been chosen so that $\chi(y)\,=w(y)-C$. Thanks to   \eqref{eq:40} and because $\overline H(0,p)>E$
$w-C$ satisfies \eqref{eq:47} in the sense of viscosity and almost everywhere in  $\overline{B_{R_1}(0)}$.
\item[(ii)] for $y\in \overline{\Omega}$, the constant~$c$ has been chosen so that $q_1\,y_1+\xi (0,q_1,y)-c\,<\, p.y$, 
  so  $\chi(y) = \min (q_1\,y_1+\xi (0,q_1,y)-c, w(y)-C)$. Both $y\mapsto q_1\,y_1+\xi (0,q_1,y)-c$ and $w-C$ are subsolutions of \eqref{eq:47} in $ \overline{\Omega}$. Hence $\chi$ is an almost everywhere  subsolution of \eqref{eq:47} in $ \overline{\Omega}$, see Lemma \ref{lemma_appendix} in the appendix.
  Note that in general, $y\mapsto p\cdot y$ is not a  subsolution of \eqref{eq:47} in $ \overline{\Omega}$.
\item[(iii)] Let $D$ be the set defined in \eqref{eq:49}.   If  $y_1\in D$, then  $\chi(y) = \min (p\cdot y, w(y)-C)$, from the choice of $R_1$. Both $y\mapsto p\cdot y$ and $w-C$ are viscosity subsolutions of \eqref{eq:47}, so $\chi$ is an almost everywhere subsolution of  \eqref{eq:47} in $D$.   Note that   $ y\mapsto q_1y_1+\xi(0,q_1,y)-c$ is generally  not a subsolution of  \eqref{eq:47} in $D$.
\item[(iv)] If $y\in \R^2 \setminus (  \overline{B_{R_1}(0)} \cup  \overline{\Omega} \cup D    )$, then $\chi(y)$ may coincide with  $w(y)-C$ or  $q_1y_1+\xi(0,q_1,y)-c$ or $p\cdot y$. In this region, the latter three functions
  are viscosity subsolutions of  \eqref{eq:47}. Hence,  $\chi$ is an almost everywhere subsolution of  \eqref{eq:47} in $\R^2 \setminus (  \overline{B_{R_1}(0)} \cup  \overline{\Omega} \cup D    )$.
\end{description}
Finally, $\chi$ is almost everywhere a subsolution of  \eqref{eq:47} in $\R^2$.  Since  $q\mapsto H(0,y,q) $ is convex, we know from~\cite[Prop. 5.1, Chapter II]{MR1484411} that this Lipschitz continuous almost everywhere subsolution is also a viscosity subsolution of  \eqref{eq:47}.

Clearly, $\chi$ also satisfies   \eqref{eq:46}.
This concludes the proof.
\end{proof}

\begin{lemma} \label{lem:3}
  Consider $p\in \R^2$.\\ If $\max( E , \overline H (0,p)) < \HoneT(0,p_1)$ or 
   $E < \overline H (0,p)= \HoneT(0,p_1)$, then there exists  a Lipschitz function $\chi:\R^2\to \R$ such that
  \begin{eqnarray}
    \label{eq:50}\chi(y)\le p\cdot y,\quad \hbox{for all } y\in \R^2,
    \\
    \label{eq:51}   H(0,y,D\chi(y))\le \HoneT(0,p_1), \quad \hbox{in }\R^2,
  \end{eqnarray}
  where \eqref{eq:51} is understood in the viscosity sense.
\end{lemma}
\begin{proof}
  Because of \eqref{eq:42},  $\HoneT(0,p_1) > E \ge \min \HoneT(0,\cdot)$. Hence, there exists $\tilde p_1\not = p_1$ such that  $ \HoneT(0,p_1)= \HoneT(0,\tilde p_1)$. 

  \begin{description}
  \item[First case:  $p_1> \tilde p_1$]

 It is possible to choose $c>0$ sufficiently large such that
\begin{equation}
  \label{eq:52}
  p_1y_1 +\xi (0,p_1,y)-c < \min \Bigl(
  p\cdot y,    \tilde p_1y_1+\overline \Pi(0,\tilde p_1) y_2\Bigr )    \quad \hbox{for all } y\in  \overline \Omega,
\end{equation}
recalling that $\xi(0,p_1,\cdot)$ and $\overline \Pi(0,\tilde p_1)$ are defined in paragraph \ref{correctors_and_cell_pb}, and in particular that
\begin{equation}
  \label{eq:53}
  H(0,y,\tilde p_1 e_1+\overline\Pi(\tilde p_1) e_2)=\HoneT(0,\tilde p_1)=\HoneT(0,p_1).
\end{equation}

Let us set $\chi_1(y)=\min(p\cdot y,  p_1y_1 +\xi (0,p_1,y)-c)$.
  Because   $p_1> \tilde p_1$,   it is  possible to choose $R_1>R_0$ such that
  \begin{eqnarray}
    \label{eq:54}
    \tilde p_1y_1+\overline \Pi(0,\tilde p_1) y_2  <
    \chi_1(y) \quad  \hbox{for all } y\in D,
  \end{eqnarray}
  where $D=\{y\in \R^2:\; y_1>0,\; |y| \ge R_1, \; |y_2|\le R_0\}$, and  $R_0$ is the fixed radius introduced in paragraph  \ref{sec:setting}, see   \eqref{eq:1}.\\
  Let us set  $\chi_2(y)=\min( \tilde p_1y_1 +\overline \Pi(0,\tilde p_1) y_2, \chi_1(y))$.
 With $w$ defined in Subsection \ref{sec:ergod-const-glob}, we may choose $C>0$ sufficiently large  such that
\begin{equation}
  \label{eq:55} w(y)-C <  \chi_2(y),  \quad \hbox{for all } y\in 
 \overline {B_{R_1}(0)}.
\end{equation}
Consider the function  $\chi: y\mapsto \min\Bigl(w(y)-C,\chi_2(y)\Bigr)$. To prove that $\chi$ is also an almost everywhere subsolution of   \eqref{eq:51},  we split the space~${\mathbb R}^2$ into four regions:
\begin{description}
\item[(i)] for $y\in \overline{B_{R_1}(0)}$,  $\chi(y)\,=w(y)-C$, and
from  \eqref{eq:40} and \eqref{eq:42},
  $w-C$ satisfies \eqref{eq:51} in the sense of viscosity and almost everywhere in  $\overline{B_{R_1}(0)}$.
\item[(ii)] for $y\in \overline{\Omega}$, the constant~$c$ has been chosen so that
   $\chi(y) = \min (p_1\,y_1+\xi (0,p_1,y)-c, w(y)-C)$. Both $y\mapsto p_1\,y_1+\xi (0,p_1,y)-c$ and $w-C$ are subsolutions of \eqref{eq:51} in $ \overline{\Omega}$. Hence $\chi$ is an almost everywhere  subsolution of \eqref{eq:51} in $ \overline{\Omega}$.
  Note that $y\mapsto p\cdot y$ and   $\tilde p_1y_1+\overline \Pi(0,\tilde p_1) y_2$   are  generally not  subsolutions of \eqref{eq:51} in $ \overline{\Omega}$.
\item[(iii)] Recall that  $D= \{y\;:\;  y_1>0, \,  |y| \ge R_1  \, |y_2|\le R_0 \}$.   If  $y\in D$, then  $\chi(y) = \min (\tilde p_1y_1+\overline \Pi(0,\tilde p_1) y_2, w(y)-C)$, from the choice of $R_1$. From \eqref{eq:53}, \eqref{eq:40} and \eqref{eq:42}, both $y\mapsto \tilde p_1y_1+\overline \Pi(0,\tilde p_1) y_2$ and $w-C$ are viscosity subsolutions of \eqref{eq:51}, so $\chi$ is an almost everywhere subsolution of  \eqref{eq:51} in $D$.   Note that  $ y\mapsto p_1y_1+\xi(0,p_1,y)-c$ may not be a subsolution of  \eqref{eq:51} in $D$.
\item[(iv)] If $y\in \R^2 \setminus (  \overline{B_{R_1}(0)} \cup  \overline{\Omega} \cup D    )$, then $\chi(y)$ may coincide with 
 $w(y)-C$,  $p_1y_1+\xi(0,p_1,y)-c$,  $ \tilde p_1y_1+\overline \Pi(0,\tilde p_1) y_2$   or $p\cdot y$. In this region, the  latter four functions are viscosity subsolutions of  \eqref{eq:51}. Hence,  $\chi$ is an almost everywhere subsolution of  \eqref{eq:51} in $\R^2 \setminus (  \overline{B_{R_1}(0)} \cup  \overline{\Omega} \cup D    )$.
\end{description}
The function
  $\chi: y\mapsto \min\Bigl(w(y)-C,\chi_2(y)\Bigr)$
 has all the desired properties.    
\item[Second case:  $p_1< \tilde p_1$]
  It is possible to choose $c>0$ sufficiently large  such that
\begin{equation*}
  \label{eq:2024bis}
  \tilde p_1y_1 +\xi (0,\tilde p_1,y)-c < p\cdot y,
  \quad \hbox{for all } y\in  \overline \Omega.
\end{equation*}
 Because   $p_1< \tilde p_1$,   it is  possible to choose $R_1>R_0$ such that
  \begin{eqnarray*}
    p \cdot y < \tilde p_1y_1 +\xi (0,\tilde p_1,y)-c,
    \quad  \hbox{for all } y\in  D
  \end{eqnarray*}
   with $D=\{y\in
    \R^2\;:
y_1>0,\; |y| \ge R_1, \; |y_2|\le R_0\}$.      

  Finally, we may choose $C>0$ sufficiently large  such that
\begin{equation*}
  \label{eq:2026bis} w(y)-C <  \min \Bigl( p\cdot y,  \tilde p_1y_1 +\xi (0,\tilde p_1,y)-c \Bigr) ,  \quad \hbox{for all } y\in 
 \overline {B_{R_1}(0)}.
\end{equation*}
The function $\chi(y)=\min \Bigl( w(y)-C , p\cdot y,  \tilde p_1y_1 +\xi (0,\tilde p_1,y)-c \Bigr)$ has all the desired properties. Indeed, we can argue exactly as in the proof of lemma \ref{lem:2} replacing $q_1$ with $\tilde p_1$, but we write the proof for clarity. To prove that $\chi$ is also an almost everywhere subsolution of   \eqref{eq:51},  we split the space~${\mathbb R}^2$ into four regions:
\begin{description}
\item[(i)] for $y\in \overline{B_{R_1}(0)}$,  $\chi(y)\,=w(y)-C$, and $w-C$ satisfies \eqref{eq:51} in the sense of viscosity and almost everywhere in  $\overline{B_{R_1}(0)}$.
\item[(ii)] for $y\in \overline{\Omega}$, the constant~$c$ has been chosen so that
   $\chi(y) = \min (\tilde p_1\,y_1+\xi (0,\tilde p_1,y)-c, w(y)-C)$. Both $y\mapsto \tilde p_1 y_1+\xi (0,\tilde p_1,y)-c$ or $w-C$ are subsolutions of \eqref{eq:51} in $ \overline{\Omega}$. Hence $\chi$ is an almost everywhere  subsolution of \eqref{eq:51} in $ \overline{\Omega}$.
  Note that $y\mapsto p\cdot y$  is not in general  a subsolution of \eqref{eq:51} in $ \overline{\Omega}$.
\item[(iii)] Set $D= \{y\;:\;  y_1>0, \,  |y| \ge R_1  \, |y_2|\le R_0 \}$.   If  $y_1\in D$, then  $\chi(y) = \min ( p\cdot y, w(y)-C)$, from the choice of $R_1$. Both $y\mapsto p\cdot y$ and $w-C$ are viscosity subsolutions of \eqref{eq:51}, so $\chi$ is an almost everywhere subsolution of  \eqref{eq:51} in $D$.   Note that  $ y\mapsto \tilde p_1y_1+\xi(0,\tilde p_1,y)-c$ is generally not  a subsolution of  \eqref{eq:51} in $D$.
\item[(iv)] If $y\in \R^2 \setminus (  \overline{B_{R_1}(0)} \cup  \overline{\Omega} \cup D    )$, then $\chi(y)$ may coincide with 
 $w(y)-C$,   $ \tilde p_1y_1++\xi(0,\tilde p_1,y)-c$   or $p\cdot y$. In this region, the  latter three functions are viscosity subsolutions of  \eqref{eq:51}. Hence,  $\chi$ is an almost everywhere subsolution of  \eqref{eq:51} in $\R^2 \setminus (  \overline{B_{R_1}(0)} \cup  \overline{\Omega} \cup D    )$.
\end{description}
  \end{description}
\end{proof}

Lemmas \ref{lem:2} and \ref{lem:3} are  used in the proofs of the following two propositions:
\begin{proposition}
  \label{sec:function--underline-2}
 For all $\phi\in C^1(\R^2)$  such that $0$ is a local minimizer of $\underline u -\phi$  and $ \overline H(0,D\phi(0))>\max (E, \HoneT(0, \partial_{1} \phi(0))$, we have
\begin{equation}
  \label{eq:57}
 \alpha \underline u(0)+ \overline H(0,D\phi(0)) \ge 0.
\end{equation}
\end{proposition}

\begin{proof}
We can always assume that $\phi(0)=\underline u(0)$ and that $\underline u -\phi$ has a strict local minimum at the origin.
For brevity, let us set $p = D\phi(0)$.  Suppose by contradiction that
\begin{equation}
  \label{eq:58}
  \alpha \underline u(0)+ \overline H(0,p) =-\theta<0,
\end{equation}
Let $\chi$ be the function arising in Lemma \ref{lem:2}. Consider the  perturbed test-function
  \begin{displaymath}
 \phi_\varepsilon(x)=\phi(x) -  p\cdot x +\varepsilon \chi(\frac x \varepsilon).
  \end{displaymath}
 Note that  $D\phi(x) -  p$ tends to $0$ as $x\to 0$ and  $\phi(x) -  p\cdot x=\underline u(0) + o(|x|)=  (-\overline H(0,p)-\theta)/\alpha   + o(|x|) $
 from \eqref{eq:58}.

 From  the definition of $\phi_\varepsilon$, the latter two points,
 \eqref{eq:46} and the regularity properties of $H$  and finally \eqref{eq:47},  
we deduce that there exists $r_0>0$ and $\varepsilon_0>0$ such that for all $0<\varepsilon<\varepsilon_0 $ and 
$0<r<r_0$,  $ \phi_\varepsilon$ is a viscosity subsolution of 
\begin{equation}
  \label{eq:59}
  \alpha \phi_\varepsilon+ H_\varepsilon ( x , D\phi_\varepsilon )\le -\frac \theta 2 \quad \hbox{ in } B_r(0).
\end{equation}
On the other hand, since $0$ is a strict local minimizer of $\underline u -\phi$, there exists $r_1>0$ and a  function   $k: (0,r_1]\to (0,1]$, 
such that $\lim_{r\to 0} k(r)=0$ and  for any $r\in (0,r_1]$,
\begin{displaymath}
  \phi(x)\le \underline u (x ) -k(r) \quad \hbox{ on } \partial B_r(0).
\end{displaymath}
From \eqref{eq:46}, we know that for $x\not =0$, $\varepsilon \chi(\frac x \varepsilon)\le p\cdot x $.

Using \eqref{eq:22}, this implies that  first fixing $r>0$ sufficiently small, we have for  $\varepsilon$ sufficiently small,
\begin{equation}
  \label{eq:60}
\phi_\varepsilon(x)    \le u_\varepsilon (x) -\frac {k(r)}  2   \quad  \hbox{ on } \partial B_r(0).
\end{equation}

From (\ref{eq:59}) and (\ref{eq:60}) and since $u_\varepsilon$ is a viscosity solution of (\ref{eq:13}), the comparison principle yields  that it is possible to choose $r>0$ such that, for $\varepsilon$ sufficiently small, $\phi_\varepsilon(x)\le  u_\varepsilon (x) -\frac {k(r)}  2$ in $\overline  {B_r(0)}$.

By choosing a sequence $\varepsilon_n$ such that $ u_{\varepsilon_n} (0)$ tends to $\underline u (0)$, we deduce that
$\phi(0)+\frac {k(r)}  2\le  \underline u(0) $, the desired contradiction.
\end{proof}

\begin{proposition}
  \label{sec:function--underline-3}
 For all $\phi\in C^1(\R^2)$  such that $0$ is a local minimizer of $\underline u -\phi$  and $  \HoneT(0,\partial_{1} \phi(0))>\max (E, \overline H(0, D\phi(0))$ or $ \HoneT(0,\partial_{1} \phi(0))=  \overline H(0, D \phi(0)>E$, we have
\begin{equation}
  \label{eq:61}
 \alpha \underline u(0)+  \HoneT(0, \partial_{1} \phi(0)) \ge 0.
\end{equation}
\end{proposition}
\begin{proof}
  The proof is identical to that of Proposition  \ref{sec:function--underline-3} except that $\chi$ is now chosen as the function appearing in Lemma \ref{lem:3}. 
\end{proof}

\begin{proposition}
  \label{sec:function--underline-4}
 For all $\phi\in C^1(\R^2)$  such that $0$ is a local minimizer of $\underline u -\phi$  and $  E \ge  \max ( \HoneT(0,\partial_{1} \phi(0)), \overline H(0, D \phi(0))$,  we have
\begin{equation}
  \label{eq:62}
 \alpha \underline u(0)+ E \ge 0.
\end{equation}
\end{proposition}

\begin{proof}
  As above, we can always assume that $\phi(0)=\underline u(0)$ and that $\underline u -\phi$ has a strict local minimum at the origin and we set $p = D\phi(0)$.
  For all $\eta>0$, it is clear from the hypothesis that  $E+\eta>  \min_{q_1\in \R}  \HoneT (0, q_1)$ and that  $E+\eta> \min_{q\in \R^2} \overline H (0, q)$.
  
  From  the latter  inequality, there exists a unique pair $(\underline q_2, \overline q_2 )\in \R^2$ such that   $\underline q_2< p_2<\overline q_2$ and that
  $\overline H (0, p_1e_1+\underline q_2 e_2)= \overline H (0, p_1e_1+\overline q_2 e_2)=E+\eta$.
  Let us set $\overline q=p_1e_1+\overline q_2 e_2$ and  $\underline q=p_1e_1+\underline q_2 e_2$. It is straightforward to check that for all $y\in \R^2$, $\min(\overline q\cdot y, \underline q\cdot y) \le p\cdot y$.
  
  There also exists $q_1\in \R$ such that $q_1> p_1$ and $\HoneT(0,q_1)=E+\eta$.
  Hence, there exists a constant $c>0$ such that $q_1y_1+\xi(0,q_1,y)-c< \min(\underline q \cdot y, \overline q \cdot y)$ for all  $y\in \overline \Omega$.
 Because $q_1>p_1$,  it is  possible to choose $R_1>R_0$ such that
  \begin{eqnarray*}
    q_1y_1+\xi(0,q_1,y)-c>    \min( \underline q \cdot y, \overline q \cdot y) , \quad  \hbox{for all } y\in  D,
  \end{eqnarray*}
  where $D=   \{y\in \R^2\;: y_1>0,\; |y| \ge R_1, \; |y_2|\le R_0  \}$ and
  $R_0$ is the fixed radius introduced in paragraph  \ref{sec:setting}, see   \eqref{eq:1}.

 We may then choose $C>0$ such that $w(y)-C <  \min(q_1y_1+\xi(0,q_1,y)-c,
 \underline q \cdot y, \overline q \cdot y)  $ for all $y\in \overline {B_{R_1}(0)}$.
  
  Collecting all the information above and arguing essentially as in the proofs of Lemmas \ref{lem:2} or \ref{lem:3}, we can check that the function
 $\chi: y\mapsto \min(  w(y)-C, q_1y_1+\xi(0,q_1,y)-c,
\underline q \cdot y, \overline q \cdot y
) $ is  a viscosity subsolution of  $H(0,y,D\chi(y))\le E+\eta$ in $\R^2$ and that $\chi(y)\le p\cdot y$ for all $y\in \R^2$.
 
 Reproducing the proof of Proposition  \ref{sec:function--underline-2} with this new  choice of $\chi$ leads to the inequality
 \begin{displaymath}
  \alpha \underline u (0) + E+\eta \ge 0.
 \end{displaymath}
Letting $\eta$ tend to $0$, we obtain the desired result.
\end{proof}

We have proved that $\underline u$ satisfies (\ref{eq:17}). 
In order to prove Theorem~\ref{sec:main-result-1}, there only remains  to establish that   $\underline u=\overline u$.

\subsection{End of the proof of Theorem~\ref{sec:main-result-1}}\label{proof_of_main_th}

We have proved that $\overline u$ is a weak viscosity subsolution of the stratified problem and that  $\underline u$ is a  viscosity supersolution of the stratified problem, see Definition \ref{def:2}.  Applying the comparison principle for stratified solutions, see \cite[theorem 19.1]{MR4704058}, we infer that $\underline u=\overline u$.  We deduce that the whole family $u_\varepsilon$ converges locally uniformly to the solution of  (\ref{eq:16}) through (\ref{eq:20}).

\section{Generalizations}
\label{sec:generalization}

\subsection{ A  longitudinal defect within a  periodic background}
\label{sec:generalization1}

\subsubsection{Setting}
\label{sec:generalization1setting}
The open set $\Omega$ is still defined by \eqref{eq:1}.
The Hamiltonian $H_\varepsilon: \R^2\times \R^2 \to \R$ is of the form  \eqref{eq:2} where $A$ is compact metric space.

We make the same assumptions on $f_\varepsilon: \R^2\times A\to \R^2$ and $\ell_\varepsilon: \R^2\times A\to \R$ except that
we now suppose that
\begin{itemize}
\item there exist   functions $f_{\per}:  \R^2\times \R^2 \times A\to \R^2$ and  $\ell_{\per}:  \R^2\times \R^2 \times A\to \R$
  that are periodic with respect to their second argument with period $[0,1]^2$,
  such that
  \begin{equation}
  \label{eq:3bis}
 f( x,y, a)=  f_{\per}(x,y,a), \quad\hbox{and} \quad  \ell( x,y, a)=  \ell_{\per}(x,y,a), \quad  \hbox{if } y\not \in \Omega.
\end{equation}
Assumption \eqref{eq:3bis} replaces assumption \eqref{eq:4}.
\item   There exist functions $f_{1,\per}:  \R^2\times \R^2\times A\to \R^2$, $(x,y,a)\mapsto f_{1,\per}(x,y,a)$ and $\ell_{1,\per}:  \R^2\times \R^2\times A\to \R$, $(x,y,a)\mapsto \ell_{1,\per}(x,y,a)$ such that
  \begin{enumerate}
  \item $f$ and $\ell$ coincide respectively with $f_{1,\per}$ and  $\ell_{1,\per}$ in the set $\R^2 \times \Bigl( (-\infty,0] \times \R \Bigr) \times A$, i.e. for all $x\in \R^2$, $y\in  (-\infty,0] \times \R$ and $a\in A$, $f(x,y,a)=f_{1,\per}(x,y,a)$ and
    $\ell(x,y,a)=\ell_{1,\per}(x,y,a)$
  \item $f_{1,\per}$ and $\ell_{1,\per}$ are $1$-periodic with respect to $y_1$
  \item for all $x\in \R^2$ and $a\in A$, $f_{1,\per}(x,y,a)= f_{\per} (x,y,a)$ and
     $\ell_{1,\per}(x,y,a)= \ell_{\per} (x,y,a)$ if $|y_2|\ge R_0$.
  \end{enumerate}
\end{itemize}
Set
\begin{displaymath}
H_{\per}(x,y,p)= \sup_{a\in A} -p\cdot f_{\per} (x,y,a) -  \ell_{\per} (x,y,a) ,
\end{displaymath}
and
\begin{displaymath}
 H_{1,\per}(x,y,p)= \sup_{a\in A} -p\cdot f_{1,\per} (x,y,a) -  \ell_{1,\per} (x,y,a) .
\end{displaymath}

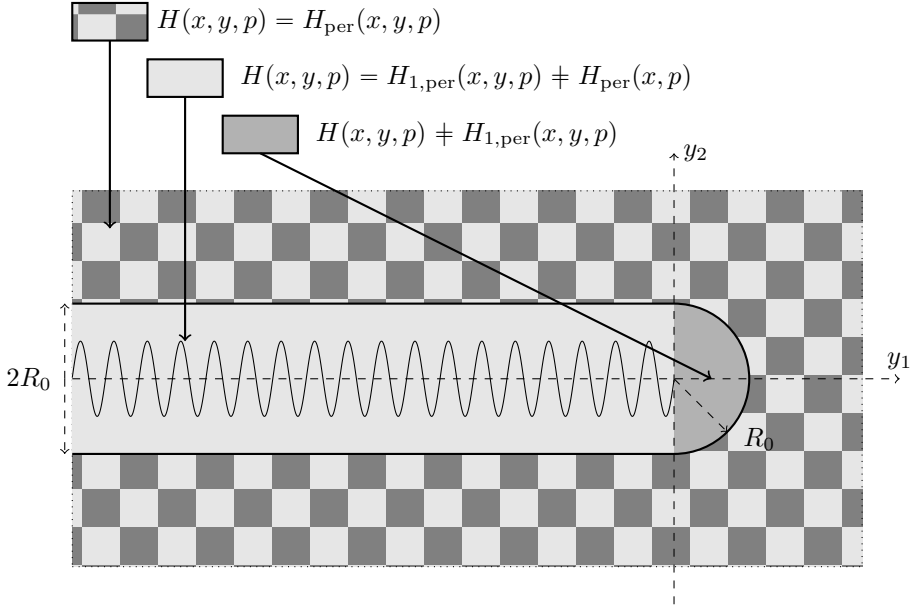
\begin{figure}[H]
  \begin{center}
    \begin{tikzpicture}

    \filldraw [dotted, pattern=flexcheckerboard_g] (-4,-2.5) rectangle (6.5,2.5);   

 \filldraw [fill=black!10,black!10] (-4,-1) rectangle (4,1);
  \draw [line width=0.8pt]  (-4,-1) -- (4,-1);
  \draw [line width=0.8pt]  (-4,1) -- (4,1);
  
  \path[fill=black!30] (4,-1)  arc (-90:90:1)  -- (4,-1);
  \path[draw,line width=0.8pt] (4,-1)  arc (-90:90:1);

   \filldraw [line width=0.8pt, pattern=flexcheckerboard_g]  (-4,5) rectangle (-3,4.5);
  \draw (-3,4.75) node[right] { {$H(x,y,p) =H_\per(x,y,p)$}};
  
   \filldraw [line width=0.8pt,fill=black!10]  (-3,4.25) rectangle (-2,3.75);
   \draw (-2,4) node[right] {{ $H(x,y,p)= H_{1,\per}(x,y,p)\not= H_\per(x,p) $}};
       \filldraw [line width=0.8pt,fill=black!30]  (-2,3.5) rectangle (-1,3);
        \draw (-1,3.25) node[right] {{ $ H(x,y,p)  \not=  H_{1,\per}(x,y,p) $}};

        \draw[->, draw opacity =1,line width=0.8pt] (-3.5,4.5) --(-3.5,2);
         \draw[->,draw opacity =1, line width=0.8pt] (-2.5,3.75) --(-2.5,0.5);
  \draw[->, draw opacity =1, line width=0.8pt] (-1.5,3) --(4.5,0);

  \draw[->, dashed] (4,-3) --(4,3);
  \draw[->,dashed] (-4,0) --(7,0);
  \draw (4,3) node[right] {{$y_2$}};
  \draw (7,0) node[above] {{$y_1$}};
  \draw[->, dashed] (-4.1,0) --(-4.1,1);
  \draw[->, dashed] (-4.1,0) --(-4.1,-1);
  \draw[->,dashed] (4,0)--( 4.7071067,-0.7071067) ;
  \draw ( 4.8,-0.8)  node[right] {{$R_0$}};
  \draw (-5,0) node[right] {{$2R_0$}};

   \draw [domain=-4:4, samples=10000] plot (\x,{0.5*sin(0.9*15.7*\x r)});
\end{tikzpicture}
\end{center}
\caption{The generic situation described in paragraph \ref{sec:generalization1setting}:  fixing $(x,p)\in \R^2\times \R^2$, the function $y\mapsto H(x,y,p)$ is continuous, periodic with period $[0,1]^2$ in the region filled with a checkerboard pattern,
and $1$-periodic with respect to $y_1$ in the lighter grey region. The set $\overline \Omega$ is the union of the grey regions. The sinusoidal graph is just meant to symbolize the fact that $y\mapsto H(x,y,p)$ is periodic with respect to $y_1$ in the lighter grey region.  Note also that  $H_{1,\per}$ is defined in the whole space $(\R^2)^3$ and that $H_{1,\per}(x,y,p)= H_\per(x,y,p)$ if $|y_2|\ge  R_0$.}
  \label{fig:2}
\end{figure}


\begin{remark}
  Examples fulfilling the assumptions above can be constructed by suitably choosing $f$ and $\ell$  in the additive form 
  $f(x,y,a)=f_0(x,y,a) + f_1(y, a)$, $\ell(x,y,a)=\ell_0 (y,a) + \ell_1(y, a)$, where $f_0: \R^2\times \R^2\times A\to \R^2$, $\ell_0 : \R^2\times \R^2\times A\to \R$,
  $f_1:  \R^2\times A\to \R^2$ and $\ell_1: \R^2\times A\to \R$ are smooth,
  $f_0$ and $\ell_0$ are periodic with respect to their second argument with period $[0,1]^2$,
  $f_1$ and $\ell_1$  vanish outside $\overline \Omega \times A$
  and coincide on $\Bigl( (-\infty,0]\times \R\Bigr) \times A$ with functions that are $1$-periodic w.r.t. $y_1$.
\end{remark}

Homogenization of the periodic Hamilton-Jacobi equation  is well understood since  \cite{LPV}.
In the periodic case, the homogenized equation  is 
\begin{equation}
  \label{eq:63}
\alpha \, u+\overline H (x,Du)=0 \quad \hbox{ in } \R^2,
\end{equation}
 where the effective Hamiltonian is characterized as follows:
 for any $p\in \R^2$, $\overline H (x,p)$ is the unique real number such that there exists a periodic corrector $\chiperp(x,\cdot)$, i.e. a viscosity solution
 $ \chiperp(x,\cdot) \in C(  \T^2)$  ($\T^2 =\R^2 /\Z^2$  denotes the torus of $\R^2$)
of 
\begin{equation}
  \label{eq:64}
  H_{\per}(x,y, p+ D_y\chiperp)=\overline H(x,p) \quad \hbox{in }\T^2.
\end{equation}
In general,  the latter periodic corrector is not unique, even up to the addition of a scalar constant.
It is well known that, under the assumptions made above, $\overline H$ inherits the properties of the periodic Hamiltonian $H_{\per}$. For any $x\in \R^2$, the function $p\mapsto \overline H(x,p)$ is convex. There exist positive constants $\overline L, \overline C, \overline c, \overline m$ and a modulus of continuity $\overline \omega$ such that,
for any $x,\tilde x \in \R^2$     $p, \tilde p\in \R^2$,
 \begin{eqnarray*}
\mid \overline H(x,p)-\overline H(x,\tilde p) \mid \le \overline L |p-\tilde p|,
\\
   \mid  \overline H(x,p)- \overline H(\tilde x,p) \mid \le \overline C|p| |x-\tilde x|+
  \overline  \omega(|x-\tilde x |),
\\
 \frac 1 {\overline c} |p|-\overline m\le \overline H(x,p)\le \overline c |p|+\overline m.
\end{eqnarray*}

\subsubsection{Effective behaviour as $\varepsilon\to 0$}
 \label{sec:generalization1main_th}
A similar strategy  as in  Section \ref{sec:proof} leads to the following theorem:
\begin{theorem}
  \label{sec:main-result-2}
We consider the solution $u_\varepsilon$ of (\ref{eq:8}).
If the  assumptions made in Subsection \ref{sec:generalization1setting} are satisfied, then,
 then,  as $\varepsilon\to 0$, the family $u_\varepsilon$ converges locally uniformly to
  a bounded and Lipschitz continuous function $u$ defined on $\R^2$, which is
  the unique solution to the  stratified problem \eqref{eq:16},  \eqref{eq:17},  \eqref{eq:18},  \eqref{eq:19},  \eqref{eq:20}
  associated to the admissible flat stratification of $\R^2$ defined in \eqref{eq:15},
  where
  \begin{itemize}
  \item the effective Hamiltonian $\overline H$ is obtained by homogenization of  the periodic problem  \eqref{eq:63} involving $H_{\per}$, see  \eqref{eq:64}
  \item  $\HoneT: \overline{\cM_1}\times \R\ \to \R$ is the  Hamiltonian defined in Section \ref{sec:proof_2} with the new version of $H_{1,\per}$
  \item  $E$ is the  effective Dirichlet datum  defined in Section~\ref{sec:ergod-const-glob}
  \end{itemize}
\end{theorem}

\begin{remark}
 The only difference with Theorem  \ref{sec:main-result-1} is the definitions of the effective Hamiltonians  $\overline H$ and $\HoneT$.
\end{remark}
  
\begin{remark}
  Note that in the assumptions, $\cM_1$ is aligned with one of the direction  of periodicity of $H_{\per}$. This is important to generalize
the construction made in Section \ref{effective_H_M1}, in particular 
the construction of $\xi(x,p_1, \cdot)$.
The result can be further generalized to the case when $\cM_1$ is  aligned with
$m e_1 +n e_2$ with $(m, n)\in \Z^2$. Our proof does not cover as such the case of irrational quotients  $m/n$. The result might hold true but we will not proceed in that direction.
\end{remark}

\begin{remark}
  The particular case when $H_{1,\per}$ coincides with  $H_{\per} $  has already been
  studied in \cite{MR4643677}. In this case, the proof proposed above becomes much simpler because it only involves $w$ (the $\xi(0,y,p)$ are irrelevant in this case).   In particular, Lemma  \ref{lem:3} becomes irrelevant and  the counterpart of Lemma \ref{lem:2} is much simpler in this case.   Note that the proof proposed in the present work is simpler compared to that in\cite{MR4643677}, because the present study of $\underline u$ only involves 
   subcorrectors instead of correctors  as in \cite{MR4643677}.
\end{remark}

\subsection{A  defect located  in a tubular neighborhood of  a straight line and longitudinally periodic  away from  the origin}
\label{sec:generalization3}

Let us focus on a case in which the fast variations of the coefficients are localized in a neighborhood of the straight line $\{x_2=0\}$.

In both regions $\{x_1<-R_0\varepsilon\}$ and  $\{x_1>R_0\varepsilon\}$, the Hamiltonian coincides respectively
with  $H_{1,-,\per}$ and $H_{1,+,\per}$ that are periodic with respect to  $x_1/\varepsilon$ (with possibly different periods).

We expect that the effective problem involves the following stratification of $\R^2$: $\R^2=\cM_0\cup\cM_1\cup\cM_2$ where $\cM_0=\{0\}$ and $\cM_1= \cM_{1,+} \cup \cM_{1,-}$,  $\cM_{1,\pm}=
\{ \pm s e_1, s> 0 \}$. Note that $\cM_{1,+}$ and  $\cM_{1,-}$ are disjoint open half-lines.

\subsubsection{Setting}
\label{sec:generalization3setting}
The description that follows is illustrated on Figure \ref{fig:3}.

 Let $\Omega_\pm$   be the two subsets of $\R^2$ defined by
\begin{eqnarray}
  \label{eq:65}
  \Omega_\pm= \{ \pm s e_1 + t e_2, s\ge  R_0, t\in (-R_0,R_0)  \},
\end{eqnarray}
where $R_0$ is a positive number.  Let us also fix  $R_1> \sqrt{2} R_0$ and set
$ \Omega=  \Omega_{+}\cup\Omega_{-}  \cup  B_{R_1}(0) $. For a small positive parameter  $\varepsilon$
that will eventually vanish, set
$\Omega_{\pm,\varepsilon}=\varepsilon \Omega_\pm$ and $\Omega_\varepsilon =\varepsilon \Omega$.
The Hamiltonian $H_\varepsilon: \R^2\times \R^2 \to \R$ is of the form  \eqref{eq:2} where $A$ is compact metric space.

We assume that the  functions  $f_\varepsilon: \R^2\times A\to \R^2$ and $\ell_\varepsilon: \R^2\times A\to \R$ have the form  \eqref{eq:3} where $f: \R^2\times\R^2\times A \to \R^2$ and $\ell: \R^2\times\R^2\times A \to \R^2$ are bounded and continuous.
The function $f$  Lipschitz continuous with respect to its  first two variables uniformly with respect to its third variable, i.e. it satisfies  \eqref{eq:flip}.
The function $\ell$  is uniformly  continuous  with respect to its  first two variables, i.e. is  satisfies  \eqref{eq:l_UC}.
  We also suppose that there exists some radius $r_f>0$ such that for any $x\in \R^2$, $y\in \R^2$, $\{f(x,y,a),a\in A\}$ contains the ball $B_{r_f}(0)$, which implies that the trajectories are locally strongly controllable.
  
  Define  $M_f$ and $M_\ell$ as in \eqref{eq:MfMl}.
We also suppose
\begin{enumerate}
\item
  there exist functions $\bar{f}:  \R^2\times A\to \R^2$ and $\bar{\ell}:  \R^2\times A\to \R$ such that  \eqref{eq:4} holds.
\item
  There exist two functions $f_{-,\per}$ and $f_{+,\per}$ from  $\R^2\times \R^2\times A$ to $\R^2$ and  two functions $\ell_{\pm,\per}:  \R^2\times \R^2\times A\to \R$, $(x,y,a)\mapsto \ell_{\pm,\per}(x,y,a)$ 
  such that
  \begin{enumerate}
  \item $f$ and $\ell$ coincide respectively  with $f_{\pm,\per}$ and $\ell_{\pm,\per}$ in the set $  \R^2 \times   \overline{\Omega_\pm}  \times A$
    i.e. for all $x\in \R^2$, $y\in \overline{\Omega_\pm}$, $a\in A$, $f(x,y,a)=f_{\pm,\per}(x,y,a)$ and $\ell(x,y,a)=\ell_{\pm,\per}(x,y,a)$
  \item $f_{\pm,\per}$ and $\ell_{\pm,\per}$ are periodic with respect to $ y_1$
    ( $f_{+,\per}$ and $\ell_{+,\per}$ have the same period, and  $f_{-,\per}$ and $\ell_{-,\per}$ have the same period)
  \item for all $x\in \R^2$ and $a\in A$, $f_{\pm,\per}(x,y,a)=\bar f (x,a)$ and
     $\ell_{\pm,\per}(x,y,a)=\bar \ell (x,a)$
    if $| y_2|\ge  R_0$.
  \end{enumerate}
\end{enumerate}

  
  

 Define $H_{1,\pm,\per}:  \R^2\times \R^2\times \R^2\to \R$ by
  \begin{displaymath}
  H_{1,\pm,\per}(x,y,p)=\max_{a\in A}  \Bigl(-p\cdot f_{\pm,\per}(x,y,a)-\ell_{\pm,\per}(x,y,a) \Bigr).    
\end{displaymath}

\begin{figure}[ht]
  \begin{center}
    \begin{tikzpicture}[scale=0.95]
     
      \draw [dotted] (-4,-2.5) rectangle (4,2.5);

      \filldraw [fill=black!10,black!10] (-4,-1) rectangle (0,1);
      \filldraw [fill=black!10,black!10] (0,-1) rectangle (4,1);
      
      \draw [line width=0.8pt]  (-4,-1) -- (4,-1);
      \draw [line width=0.8pt]  (-4,1) -- (4,1);

      \draw [domain=-4:0, samples=10000] plot (\x,{0.5*sin(5*\x r)});
       \draw [domain=0:4, samples=10000] plot (\x,{0.5*sin(8*\x r)});

       \filldraw [line width=0.8pt,fill=black!30] (0,0) circle(1.25);

  \draw [line width=0.8pt]  (-2,6) rectangle (-1,5.5);
   \draw (-1,5.75) node[right] {{$H(x,y,p)=\overline{H}(x,p)$}};

   \filldraw [line width=0.8pt,fill=black!10]  (-1,5.25) rectangle (0,4.75);
       \draw (-0,5) node[right] {{$H(x,y,p)= H_{1,-,\per}(x,y,p)\not= \overline{H}(x,p) $}};
       \filldraw [line width=0.8pt,fill=black!30]  (0,4.5) rectangle (1,4);
       \draw (1,4.25) node[right] {{ $ H(x,y,p)  \not=  H_{1,\pm,\per}(x,y,p) $}};

        \filldraw [line width=0.8pt,fill=black!10]  (1,3.75) rectangle (2,3.25);
        \draw (2,3.5) node[right] {{$H(x,y,p)= H_{1,+,\per}(x,y,p)\not= \overline{H}(x,p) $}};

        \draw[->, draw opacity =1,line width=0.8pt] (-1.5,5.5) --(-3.5,2);
         \draw[->,draw opacity =1, line width=0.8pt] (-0.5,4.75) --(-2.5,0.5);
         \draw[->, draw opacity =1, line width=0.8pt] (0.5,4) --(0.5,0);
          \draw[->,draw opacity =1, line width=0.8pt] (1.5,3.25) --(2.5,0.5);

  \draw[->, dashed] (0,-3) --(0,3);
  \draw[->,dashed] (-4,0) --(5,0);
  \draw (0,3) node[right] {{$y_2$}};
  \draw (5,0) node[above] {{$y_1$}};
  \draw[->, dashed] (-4.1,0) --(-4.1,1);
  \draw[->, dashed] (-4.1,0) --(-4.1,-1);
  \draw[->,dashed] (0,0)--( 0.88388,-0.88388) ;
  \draw ( 0.9,-0.9)  node[right] {{$R_1$}};
  \draw (-5,0) node[right] {{$2R_0$}};

\end{tikzpicture}
\end{center}
\caption{The generic situation described in paragraph \ref{sec:generalization3setting}
: fixing $(x,p)\in \R^2\times \R^2$, the function $y\mapsto H(x,y,p)$ is continuous, constant in the white region and periodic with respect to $y_1$ in the two  regions in light grey. The sinusoidal graphs are just meant to symbolize the fact that $y\mapsto H(x,y,p)$ is periodic with respect to $y_1$ in the two regions in light grey.  Note also that $H_{1,\pm,\per}(x,y,p)= \overline H(x,p)$ if $|y_2|\ge  R_0$.}
  \label{fig:3}
\end{figure}
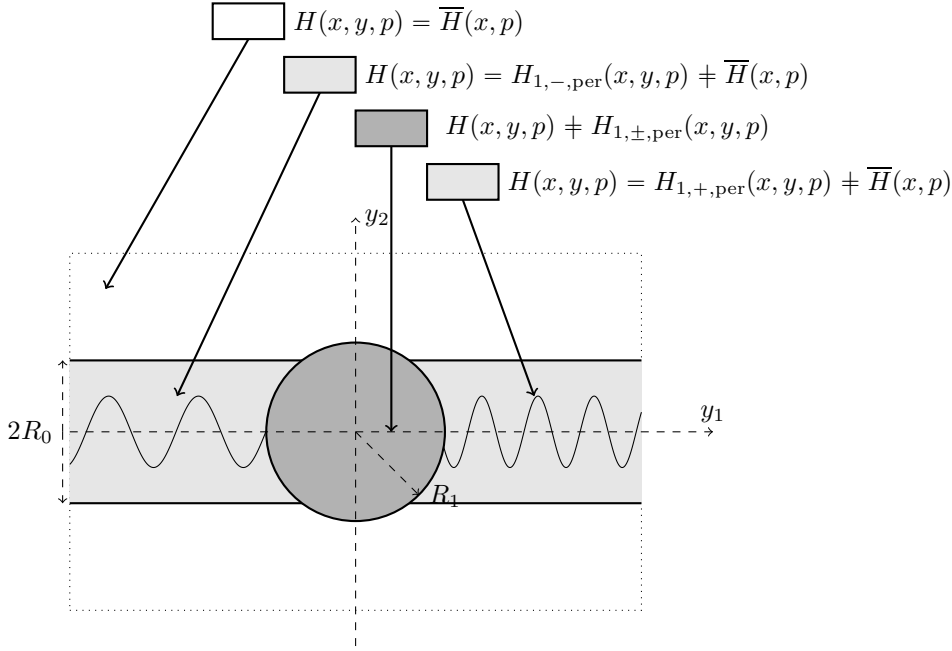

\subsubsection{Effective behaviour as $\varepsilon\to 0$}
\label{sec:generalization1main_th3}

 Let  $H_{1,\pm}: \overline \cM_{1,\pm}\times \R\to \R$ be the effective Hamiltonian obtained by reproducing the
 analysis in Section \ref{sec:proof_2} and replacing $\cM_1$ in  Section \ref{sec:proof_2}  by $\cM_{1,\pm}$, and $H_{1,\per}$ by $ H_{1,\pm,\per}$.

\begin{theorem}
  \label{sec:main-result-3}
  We consider the solution $u_\varepsilon$ of (\ref{eq:8}).
  If the  assumptions made in paragraph \ref{sec:generalization3setting} are satisfied,  then,
  as $\varepsilon\to 0$, the family $u_\varepsilon$ converges locally uniformly to
  a bounded and Lipschitz function $u$ defined on $\R^2$, which is
  the unique solution to the following stratified problem
  associated to the admissible flat stratification  $\R^2=\cM_0\cup\cM_1\cup\cM_2$ :
  \begin{enumerate}
  \item $u$ is a viscosity solution of
    \begin{equation}
      \label{eq:68}
      \alpha \, u+\overline H(\cdot,Du)= 0\quad  \hbox{in }\cM_2.
    \end{equation}
  \item
    \begin{enumerate}
    \item If  $\phi\in C^1(\R^2)$ is such that $u-\phi$ has a local minimum at $x\in \cM_{1,+}$,  then
      \begin{equation}\label{eq:18816+}
        \alpha \, u(x)+\max\left( \HoneplusT(x, \partial_1\phi(x)),  \overline H(x,D\phi(x))\right)\ge 0,
      \end{equation}
 and  if  $\phi\in C^1(\R^2)$ is such that $u-\phi$ has a local minimum at  $x\in \cM_{1,-}$,  then
       \begin{equation}\label{eq:18816-}
        \alpha \, u(x)+\max\left( \HoneminusT(x, \partial_1\phi(x)),  \overline H(x,D\phi(x))\right)\ge 0,
      \end{equation}
      where $H_{1,\pm}: \overline{\cM_{1,\pm}}\times \R\ \to \R$ are the  Hamiltonians defined immediately above.
    \item  If  $\phi\in C^1(\cM_{1,+})$ is such that $u-\phi$ has a local maximum at $x\in \cM_{1,+}$, then
      \begin{equation} \label{eq:18817+}
        \alpha \, u(x)+ \HoneplusT(x,  \phi'(x)) \le 0,
      \end{equation}
      and if
      $\phi\in C^1(\cM_{1,-})$ is such that $u-\phi$ has a local maximum at $x\in \cM_{1,-}$, then
      \begin{equation} \label{eq:18817-}
        \alpha \, u(x)+ \HoneminusT(x,  \phi'(x)) \le 0.
      \end{equation}
    \end{enumerate}
  \item
    \begin{enumerate}
    \item
      If  $\phi\in C^1(\R^2)$ is such that $u-\phi$ has a local minimum at $0$, then
      \begin{equation}\label{eq:69}
        \alpha \, u(0)+\max\left( E, \HoneplusT(0,\partial_1 \phi(0)), \HoneminusT(0,\partial_1\phi(0)),  \overline H(0,D\phi(0))\right)\ge 0,
      \end{equation}
      where $E$ is the  effective Dirichlet datum  defined in Section~\ref{sec:ergod-const-glob}.
    \item
      \begin{equation}\label{eq:70}
        \alpha \, u(0)+E \le  0.
      \end{equation}
    \end{enumerate}
  \end{enumerate}
\end{theorem}
\paragraph{Elements of proof}
The strategy of proof is exactly the same as in Section \ref{sec:proof} above. The main changes concern the technical lemmas contained in Subsection \ref{sec:function--underline} and devoted to the construction of subcorrectors. The following lemma is the counterpart of Lemma   \ref{lem:2}:
\begin{lemma} \label{lem:182}
  Consider $p\in \R^2$. If $    \max(E,\HoneplusT (0,p_1), \HoneminusT (0,p_1))<\overline H (0,p)$, then there exists
  a Lipschitz function $\chi:\R^2\to \R$ such that
  \begin{eqnarray}
    \label{eq:71}
\chi(y)\le p\cdot y,\quad \hbox{for all } y\in \R^2,
    \\
    \label{eq:72}   H(0,y,D\chi(y))\le \overline H(0,p), \quad \hbox{in }\R^2,
  \end{eqnarray}
  where \eqref{eq:72} is understood in the viscosity sense.
\end{lemma}
\begin{proof}
  There exists a unique $q_{1,-}>p_1$ such that $\HoneminusT(0,q_{1,-})=\overline H(0,p)$ and a
  unique $q_{1,+}<p_1$ such that $\HoneplusT(0,q_{1,+})=\overline H(0,p)$.
  Because $q_{1,-}>p_1>q_{1,+}$,
  there  exists a constant $c>0$ such that $q_{1,-} y_1+\xi_-(0,q_{1,-},y)-c< p\cdot y$ for all $y\in\overline \Omega_-$ and that
  $q_{1,+} y_1+\xi_+(0,q_{1,+},y)-c< p\cdot y$ for all $y\in\overline \Omega_+$. Here $\xi_\pm$ are the counterparts of
  $\xi$ defined in  Subsection \ref{correctors_and_cell_pb} when $ H_{1,\per}$ is replaced by $H_{1,\pm,\per}$.
  
 Next, there exists $R_2>R_1$ (recall that $R_1$ is fixed in paragraph \ref{sec:generalization3setting}) such that
  \begin{displaymath}
    q_{1,+}y_1+\xi_+(0,q_1,y)  < q_{1,-}y_1+\xi_-(0,q_1,y) ,\quad \hbox{for all } y\in
D_+
\end{displaymath}
  and that
\begin{equation*}
  q_{1,-}y_1+\xi_-(0,q_1,y)  < q_{1,+}y_1+\xi_+(0,q_1,y) ,\quad \hbox{for all } y\in
  D_-,
\end{equation*}
where
\begin{eqnarray}
  \label{eq:73}
 D_+=\{y\in \R^2:\;
  y_1>0,\; |y| \ge R_2, \; |y_2|\le R_0   \},
  \\
   \label{eq:74}
  D_-=\{y\in \R^2:\;
y_1<0,\; |y| \ge R_2, \; |y_2|\le R_0   \}.
\end{eqnarray}

Note that
$\overline H (0,  q_{1,\pm} e_1+D\xi_{\pm}(0,q_{1,\pm},y) ) = \overline H(0,p)$ in $\{y: |y_2|>R_0\}$.

With $w$ defined in the same way as in Subsection \ref{sec:ergod-const-glob},
we may then choose $C>0$ such that $w(y)-C <  \min(q_{1,-}y_1+\xi_-(0,q_{1,-},y)-c, q_{1,+}y_1+\xi_+(0,q_{1,+},y)-c, p\cdot y)  $  for all $y\in \overline { B_{R_2}(0)}$.
The function $\chi: y\mapsto \min(  w(y)-C, q_{1,-}y_1+\xi_-(0,q_{1,-},y)-c, q_{1,+}y_1+\xi_+(0,q_{1,+},y)-c, p\cdot y) $ is Lipschitz continuous  as the minimum of four Lipschitz continuous functions. To prove that $\chi$ is an  almost everywhere subsolution of   \eqref{eq:72},  we split the space~${\mathbb R}^2$ into four regions:
\begin{description}
\item[(i)] for $y\in \overline{B_{R_2}(0)}$, $\chi(y)\,=w(y)-C$, and $w-C$ satisfies \eqref{eq:72} in the sense of viscosity and almost everywhere in  $\overline{B_{R_2}(0)}$.
\item[(ii)]   In $D_-$,  the set defined by  \eqref{eq:74},
  $\chi(y)= \min(q_{1,-}y_1+\xi_-(0,q_{1,-},y)-c, w(y)-C)$.
  Both  $y\mapsto q_{1,-}y_1+\xi_-(0,q_{1,-},y)-c$ and $w-C$ are subsolutions of \eqref{eq:72} in $D_-$, hence  $\chi$ is an almost everywhere  subsolution of \eqref{eq:72} in $ D_-$.
\item[(iii)] Similarly,  in $D_+$, the set defined by \eqref{eq:73}, $\chi(y)= \min(q_{1,+}y_1+\xi_+(0,q_{1,+},y)-c, w(y)-C)$.
  Both  $y\mapsto q_{1,+}y_1+\xi_+(0,q_{1,+},y)-c$ and $w-C$ are subsolutions of \eqref{eq:72} in $D_+$, hence  $\chi$ is an almost everywhere  subsolution of \eqref{eq:72} in $ D_+$.
\item[(iv)] If $y\in \R^2 \setminus (  \overline{B_{R_2}(0)} \cup  D_- \cup D_+   )$, then $\chi(y)$ may coincide with either  $w(y)-C$, $q_{1,-}y_1+\xi_-(0,q_{1,-},y)-c$, $q_{1,+}y_1+\xi_+(0,q_{1,+},y)-c$ or $p\cdot y$.
 In this region, the latter four functions
  are viscosity subsolutions of  \eqref{eq:72}. Hence,  $\chi$ is an almost everywhere subsolution of  \eqref{eq:72} in $\R^2 \setminus (  \overline{B_{R_2}(0)} \cup D_- \cup D_+    )$.
\end{description}
Finally, $\chi$ is almost everywhere a subsolution of  \eqref{eq:72} in $\R^2$.  Since  $q\mapsto H(0,y,q) $ is convex, we know from~\cite[Prop. 5.1, Chapter II]{MR1484411} that this Lipschitz continuous almost everywhere subsolution is also a viscosity subsolution of  \eqref{eq:72}. Clearly, $\chi$ satisfies  also \eqref{eq:71}.
\end{proof}

 The following lemma is the counterpart of Lemma   \ref{lem:3}:
\begin{lemma} \label{lem:183}
  Consider $p\in \R^2$. 
  \begin{enumerate}
  \item
    If
    \begin{equation}
      \label{eq:75}
      \max( E , \overline H (0,p)) < \max( \HoneminusT(0,p_1) ,  \HoneplusT(0,p_1)),
    \end{equation}
    or
    \begin{eqnarray}  \label{eq:76}
      E < \overline H (0,p)=  \max( \HoneminusT(0,p_1) ,  \HoneplusT(0,p_1)),\\
       \label{eq:77}
      \HoneminusT(0,p_1) \not =  \HoneplusT(0,p_1),
    \end{eqnarray}
    or 
    \begin{eqnarray}  \label{eq:78}
     E < \overline H (0,p)=  \HoneminusT(0,p_1) =  \HoneplusT(0,p_1)),\\
      \label{eq:79}
      \hbox{  near $p_1$, $\HoneminusT (0,\cdot)$ or  $-\HoneplusT(0,\cdot)$ is decreasing},
    \end{eqnarray}
 then there exists  a Lipschitz function $\chi:\R^2\to \R$ such that
    \begin{eqnarray}
      \label{eq:80}\chi(y)\le p\cdot y,\quad \hbox{for all } y\in \R^2,
      \\
      \label{eq:81}   H(0,y,D\chi(y))\le  \max( \HoneminusT(0,p_1) ,  \HoneplusT(0,p_1)), \quad \hbox{in }\R^2,
    \end{eqnarray}
    where \eqref{eq:81} is understood in the viscosity sense.

  \item   If  \eqref{eq:78} holds and both  $\HoneminusT(0,
    \cdot)$ and  $-\HoneplusT(0,\cdot)$
    are non decreasing 
    near $p_1$,
    then for  $\eta>0$,
 there exists  a Lipschitz function $\chi:\R^2\to \R$ that satisfies  \eqref{eq:80} and  
    \begin{equation}
      \label{eq:82}   H(0,y,D\chi(y))\le  \max( \HoneminusT(0,p_1) ,  \HoneplusT(0,p_1))+\eta \quad \hbox{in }\R^2.
    \end{equation}
  \end{enumerate}
\end{lemma}
\begin{remark}
    From the convexity of  $ \HoneminusT(0,\cdot)$ and $ \HoneplusT(0,\cdot)$, we deduce
   that if   $E<   \HoneminusT(0,p_1) =
    \HoneplusT(0,p_1))$, then  $\HoneminusT(0,\cdot)$ and  $\HoneplusT(0,\cdot)$
    are  (strictly) monotone near $p_1$, because
    \begin{eqnarray*}
      \HoneminusT(0,p_1)>E\ge \min_{q_1}\HoneminusT(0,q_1),\\
      \HoneplusT(0,p_1)>E\ge \min_{q_1}\HoneplusT(0,q_1).      
    \end{eqnarray*}
 This justifies the assumptions \eqref{eq:78}-\eqref{eq:79} and the assumptions in case 2.
  \end{remark}
\begin{proof}
  We start with proving point 1.
  \begin{description}
  \item[First case:  $ \HoneminusT(0,p_1)> \HoneplusT(0,p_1)$.]
    $\;$\\
 Because  $\HoneminusT(0,p_1) > E \ge \min \HoneplusT(0,\cdot)$, we see that there exists $\tilde p_1< p_1$ such that  $  \HoneplusT(0,\tilde p_1)=\HoneminusT(0,p_1)$.
  It is possible to choose $c>0$ sufficiently large  such that
  \begin{equation}
  \label{eq:83}
  p_1y_1 +\xi_{-} (0,p_1,y)-c < \min \Bigl(
  p\cdot y,    \tilde p_1y_1+\xi_+ (0,\tilde p_1,y)\Bigr )    \quad \hbox{for all } y\in  \overline \Omega_-.
\end{equation}
Let us set $\chi_1(y)=\min(p\cdot y,  p_1y_1 +\xi_- (0,p_1,y)-c)$.  Because   $p_1> \tilde p_1$,   it is  possible to choose $R_2>R_1$ such that
  \begin{eqnarray}
    \label{eq:84}
    \tilde p_1y_1+\xi_+ (0,\tilde p_1,y)  <
    \chi_1(y) \quad  \hbox{for all } y\in \R^2 \hbox{ such that }
    \left \{
      \begin{array}[c]{l}
y_1>0,\\ |y| \ge R_2, \\ |y_2|\le R_0,       
      \end{array}\right.
  \end{eqnarray}
  recalling that $R_0$ and $R_1$ were introduced in paragraph \ref{sec:generalization3setting}.\\
    Let us set  $\chi_2(y)=\min( \tilde p_1y_1 +\xi_+ (0,\tilde p_1,y), \chi_1(y))$.
  With $w$ defined in Subsection \ref{sec:ergod-const-glob}, we may choose $C>0$ sufficiently large such that
\begin{equation}
  \label{eq:85} w(y)-C <  \chi_2(y),  \quad \hbox{for all } y\in 
 \overline {B_{R_2}(0)}.
\end{equation}
In view of  \eqref{eq:83},  \eqref{eq:84},  \eqref{eq:85},
 the function
  $\chi: y\mapsto \min\Bigl(w(y)-C,\chi_2(y)\Bigr)$
 has all the desired properties.
     \item[Second  case:  $ \HoneplusT(0,p_1)> \HoneminusT(0,p_1)$.]
 The desired result is obtained in a very similar manner as in the previous case, by exchanging the roles of $+$ and $-$.
\item[Third  case:  $ \HoneplusT(0,p_1)= \HoneminusT(0,p_1) > \overline H(0,p)  $.] $\;$\\
    Clearly, $ \HoneplusT(0,p_1)> \overline H(0,p)$ implies that
    $
       \HoneplusT(0,p_1)> \min_{q}\overline H(0,p_1e_1+q e_2)$, hence the slopes $ \underline \Pi_{\pm}$ and  $\underline \Pi_{\pm}$ can be defined similarly as in \eqref{eq:26} and \eqref{eq:27} and 
     \begin{equation}
       \label{eq:86}
       \underline \Pi_+(0, p_1) <  p_2 < \overline \Pi_+(0, p_1).
     \end{equation}
     Note that 
 \begin{equation}
   \label{eq:87}
    \underline \Pi_-(0, p_1) =  \underline \Pi_+(0, p_1) \quad \hbox{and } \quad    \overline \Pi_-(0, p_1) =  \overline \Pi_+(0, p_1),
  \end{equation}
  because  $ \HoneplusT(0,p_1)= \HoneminusT(0,p_1) $.
  From the properties of $\xi_{\pm}$, we deduce easily that there exists $c>0 $ such that
  \begin{equation}
    \label{eq:88}
    \begin{array}[c]{rcll}
      p_1y_1+ \xi_+(0,p_1, y) -c &<& p\cdot y,\quad\quad   &\hbox{if } y\in  \overline{\Omega_+},\\
      p_1y_1+ \xi_-(0,p_1, y) -c &<& p\cdot y,\quad\quad   &\hbox{if } y\in  \overline{\Omega_-}.   
    \end{array}
  \end{equation}
  Because of \eqref{eq:86} and \eqref{eq:87}, there exists $\delta>0$ such that
\begin{equation}
  \label{eq:89}
    p\cdot y <\left\{ 
    \begin{array}[c]{ll}
     p_1y_1+ \xi_+(0,p_1, y) -c \quad\quad   & \hbox{ if } y_1\ge 0 \hbox{ and  }  |y_2|>\delta,
\\                                            
         p_1y_1+ \xi_-(0,p_1, y) -c,\quad\quad & \hbox{ if } y_1\le 0 \hbox{ and  }  |y_2|>\delta.
    \end{array}\right.
  \end{equation}
  Let us set $R_2=\max(\delta, R_1)$.  With $w$ defined in Subsection \ref{sec:ergod-const-glob}, we may choose $C>0$ sufficiently large such that
  \begin{equation}
     \label{eq:90}
    w(y)-C < \left\{ \begin{array}[c]{ll}
     \min ( p\cdot y,   p_1y_1+ \xi_+(0,p_1, y) -c ) \quad &\hbox{if } y\in 
                                                                       \overline {B_{R_2}(0)} \hbox{ and  } y_1\ge 0,\\
       \min ( p\cdot y,   p_1y_1+ \xi_-(0,p_1, y) -c ) \quad &\hbox{if } y\in 
 \overline {B_{R_2}(0)} \hbox{ and  } y_1\le 0.
    \end{array}\right.
  \end{equation}
Let us define
\begin{equation}
  \label{eq:91}
  \chi (y) =\left\{
    \begin{array}[c]{ll}
      \min (w(y)-C,  p\cdot y,   p_1y_1+ \xi_-(0,p_1, y) -c)  \quad &\hbox{if } y_1\le 0,\\
       \min (w(y)-C,  p\cdot y,   p_1y_1+ \xi_+(0,p_1, y) -c)  \quad &\hbox{if } y_1\ge 0.
    \end{array}
\right . 
\end{equation}
Thanks to  \eqref{eq:89} and  \eqref{eq:90}, $\chi$ is Lipschitz continuous in $\R^2$ and $\chi(y)\le  p\cdot y$ at all $y\in \R^2$. It can also be checked that $\chi$ is viscosity subsolution of  \eqref{eq:80}.
\item[Fourth case:   \eqref{eq:78} and  \eqref{eq:79} hold.]
  We may assume without loss of generality that  $\HoneplusT(0,\cdot)$ is increasing near $p_1$.   \\
    There exists $\tilde p_1 < p_1$ such that $E< \HoneplusT(0,\tilde p_1)<
  \HoneplusT(0, p_1)$.  Hence, there exists a constant $R_2>R_1$ such
  \begin{eqnarray*}
    \tilde p_1 y_1 +\xi_+(0,\tilde p_1,y) <    p_1 y_1 +\xi_-(0, p_1,y) \quad \hbox{if }
    |y|\ge R_2, \; y_1> 0   \hbox{ and } |y_2|<R_0,\\
     \tilde p_1 y_1 +\xi_+(0,\tilde p_1,y) >    p_1 y_1 +\xi_-(0,p_1,y) \quad \hbox{if }  |y|\ge R_2, \; y_1< 0   \hbox{ and } |y_2|<R_0 .
  \end{eqnarray*}
  Then there exists a constant $c>0$ such that
  \begin{displaymath}
    \min (  \tilde p_1 y_1 +\xi_+(0,\tilde p_1,y) -c ,  p_1 y_1 +\xi_-(0,p_1,y)-c )< p\cdot y \quad \hbox{in } \overline \Omega.
  \end{displaymath}
  Finally, there exists $C>0$ such that
  \begin{displaymath}
    w(y)-C \le \min ( \tilde p_1 y_1 +\xi_+(0,\tilde p_1,y) -c ,
    p_1 y_1 +\xi_-(0,p_1,y)-c, p\cdot y ) \quad \hbox{in }  B_{R_2}(0).
  \end{displaymath}
  The function
  \begin{displaymath}
    \chi(y)=  \min ( w(y)-C, \tilde p_1 y_1 +\xi_+(0,\tilde p_1,y) -c ,
    p_1 y_1 +\xi_-(0,p_1,y)-c, p\cdot y )
  \end{displaymath}
  has all the desired properties.
\end{description}
Point 1. is proved.

We now tackle point 2.  Since  $\HoneplusT(0,\cdot)$ is non increasing near $p_1$, coercive and convex,
there exists  $\tilde p_1 < p_1$ such that $E< \HoneplusT(0,\tilde p_1)=
\HoneplusT(0, p_1)+\eta$. The function $\chi$ constructed in the Fourth case above with the new value of $\tilde p_1$ satisfies \eqref{eq:80} and \eqref{eq:82}.
\end{proof}

The following proposition is the counterpart of  Proposition   \ref{sec:function--underline-3}. Its proof is essentially the same as that of  Proposition   \ref{sec:function--underline-3}, except in a particular case.
\begin{proposition}
  \label{sec:}
  For all $\phi\in C^1(\R^2)$  such that $0$ is a local minimizer of $\underline u -\phi$  and
either 
\begin{displaymath}
  \max(E, \overline H(0, D \phi(0))) <  \max ( \HoneminusT(0,\partial_{1} \phi(0)),
  \HoneplusT(0,\partial_{1} \phi(0)))
\end{displaymath}
or
\begin{displaymath}
E< \overline H(0, D \phi(0)))  = \max ( \HoneminusT(0,\partial_{1} \phi(0)),
  \HoneplusT(0,\partial_{1} \phi(0)))  ,
\end{displaymath}
we have
\begin{equation}
  \label{eq:1875_a}
 \alpha \underline u(0)+ \max ( \HoneminusT(0,\partial_{1} \phi(0)),
  \HoneplusT(0,\partial_{1} \phi(0))) \ge 0.
\end{equation}
\end{proposition}
\begin{proof}
  Let us set $p= D\phi(0)$.

  If $p$ satisfies the assumptions in point 1. of Lemma  \ref{lem:183}, then the proof follows exactly the same arguments as the proof of Proposition
  \ref{sec:function--underline-3}, using the subcorrector arising in  point 1. of Lemma  \ref{lem:183}.
  
  If $p$ satisfies the assumptions in point 2. of  Lemma  \ref{lem:183}, then,
  using the approximate subcorrector $\chi$ arising in  point 2. of  Lemma  \ref{lem:183}, the same arguments yield that  for any $\eta>0$ small enough,
  \begin{displaymath}
   \alpha \underline u(0)+ \max ( \HoneminusT(0,p_1),
  \HoneplusT(0,p_1)) \ge -\eta,    
  \end{displaymath}
and  \eqref{eq:1875_a} is obtained by letting $\eta$ tend to $0$.
\end{proof}

`
The following proposition is the counterpart of  Proposition   \ref{sec:function--underline-4}. A key argument in its proof  involves 
subcorrectors that are different from those constructed in  Lemmas \ref{lem:2} and \ref{lem:3} above.
\begin{proposition}
  \label{sec:function--underline-184}
  For all $\phi\in C^1(\R^2)$  such that $0$ is a local minimizer of $\underline u -\phi$  and $  E \ge  \max ( \HoneminusT(0,\partial_{1} \phi(0)),
   \HoneplusT(0,\partial_{1} \phi(0)),\overline H(0, D \phi(0))$,  we have
\begin{equation}
  \label{eq:92}
 \alpha \underline u(0)+ E \ge 0.
\end{equation}
\end{proposition}

\begin{proof}
  We may  assume that $\phi(0)=\underline u(0)$ and that $\underline u -\phi$ has a strict local minimum at the origin and we set $p = D\phi(0)$.
  For all $\eta>0$, it is clear that  $E+\eta>  \min_{q_1\in \R}  H_{1,\pm} (0, q_1)$ and that  $E+\eta> \min_{q\in \R^2} \overline H (0, q)$.
  From  the latter  inequality, there exists a unique pair $(\underline q_2, \overline q_2 )\in \R^2$ such that   $\underline q_2< p_2<\overline q_2$ and that
  $\overline H (0, p_1e_1+\underline q_2 e_2)= \overline H (0, p_1e_1+\overline q_2 e_2)=E+\eta$.
  Let us set $\overline q=p_1e_1+\overline q_2 e_2$ and  $\underline q=p_1e_1+\underline q_2 e_2$. It is straightforward to check that for all $y\in \R^2$, $\min(\overline q\cdot y, \underline q\cdot y) \le p\cdot y$. 
  There also exist $q_{1,-}, q_{1,+}\in \R$ such that $q_{1,-}> p_1>q_{1,+}$ and $\HoneminusT(0,q_{1,-})= \HoneplusT(0,q_{1,+})=E+\eta$.
  Hence, there exists a constant $c>0$ such that $q_{1,-}y_1+\xi_-(0,q_{1,-},y)-c< \min(\underline q \cdot y, \overline q \cdot y)$ for all  $y\in \overline {\Omega_-}$ and  $q_{1,+}y_1+\xi_+(0,q_{1,+},y)-c< \min(\underline q \cdot y, \overline q \cdot y)$ for all  $y\in \overline {\Omega_+}$.
 It is  then possible to choose $R_2>R_1$ such that
  \begin{eqnarray*}
    q_{1,-}y_1+\xi_-(0,q_{1,-},y)-c>    \min( \underline q \cdot y, \overline q \cdot y, q_{1,+}y_1+\xi_+(0,q_{1,+},y)-c   ) ,
  \end{eqnarray*}
for all  $y\in \R^2$  such that $ y_1>0$, $|y| \ge R_2$ and $|y_2|\le R_0$,       
and  
 \begin{eqnarray*}
   q_{1,+}y_1+\xi_+(0,q_{1,+},y)-c >    \min( \underline q \cdot y, \overline q \cdot y, q_{1,-}y_1+\xi_-(0,q_{1,-},y)-c ) ,
 \end{eqnarray*}
for all  $y\in \R^2$  such that $ y_1<0$, $|y| \ge R_2$ and $|y_2|\le R_0$.
Recall that  $R_0$, $R_1$ are  introduced in paragraph \ref{sec:generalization3setting}.

Let us set
\begin{displaymath}
  \chi_1(y)= \min( \underline q \cdot y, \overline q \cdot y, q_{1,-}y_1+\xi_-(0,q_{1,-},y)-c,  q_{1,+}y_1+\xi_+(0,q_{1,+},y)-c ). 
\end{displaymath}
 We may then choose $C>0$ such that $w(y)-C <  \chi_1(y)$ for all $y\in \overline {B_{R_2}(0)}$.
  
 Collecting all the information above and arguing as above, the function
 $\chi: y\mapsto \min(  w(y)-C,\chi_1(y) ) $ is  a subsolution of  $H(0,y,D\chi(y))\le E+\eta$ in $\R^2$ and $\chi(y)\le p\cdot y$ for all $y\in \R^2$.
 
 Reproducing the proof of Proposition  \ref{sec:function--underline-2} with this new  choice of $\chi$ leads to the inequality
 \begin{displaymath}
  \alpha \underline u (0) + E+\eta \ge 0.
 \end{displaymath}
Letting $\eta$ tend to $0$, we obtain the desired result.
\end{proof}

\subsection{Perspectives}
The most natural generalization that we do not tackle here
concerns longitudinally periodic defects located near two half-lines with a common endpoint, forming an angle different from $0$ and $\pi$. More explicitly, consider two linearly independent unitary vectors $b_1$ and $b_2$ in $\R^2$. For $i=1,2$,  let $\Omega_i$ be the subset of $\R^2$ defined by
\begin{displaymath}
  \Omega_i= \{ s b_i + t b_i^\perp, s\le 0, t\in (-R_0,R_0)  \},
\end{displaymath}
where  $b_i^\perp$  is a unitary vector orthogonal to $b_i$ and $R_0$ is a positive number. Let us also fix  $R_1>R_0$  such that $ \bigcap_{i=1}^2 \overline {\Omega_i}\subset  B_{R_1}(0)$ and set $\Omega=
B_{R_1}(0) \cup 
\bigcup_{i=1}^2 \Omega_i $.
For a small positive parameter  $\varepsilon$
that will eventually vanish, 
the Hamiltonian $H_\varepsilon: \R^2\times \R^2 \to \R$ is a smooth function for simplicity, convex and coercive w.r.t. its second argument,  and of the form  \eqref{eq:7}, where
\begin{itemize}
\item $H(x,y,p)=\overline{H}(x,p)$  if $y\notin \Omega$
\item for $y\in  \Omega_i$ and far enough from the origin, $H(x, y,p)= H_{i,\per}(x,y,p)$, where $ H_{i,\per}$ is periodic with respect to $y \cdot b_i\in \R$.
\end{itemize}
We expect that the effective problem involve the stratification  $\R^2=\cM_0\cup\cM_1\cup\cM_2$ where $\cM_0=\{0\}$ and $\cM_1=\cup_{i=1}^2 \cM_{1,i}$,  $\cM_{1,i}=\{ s b_i, s< 0 \}$.

Although this situation somehow resembles the one addressed in Subsection \ref{sec:generalization3}, the homogenization should be more challenging.
Indeed, since the angle between  $\cM_{1,1}$ and $\cM_{1,2}$ is not $\pi$, it does not seem easy to handle the interactions between the correctors $\xi_i$ associated to the effective tangential Hamiltonians $ \overline H_{i,T}$, $i=1,2$, as we did  in Subsection \ref{sec:generalization3}.

A strategy that would be succesful in the latter case should also work with no major changes
for  longitudinally periodic defects located near $N$ half-lines with a common endpoint, $N$ being  an arbitrary positive integer.


\appendix
\section{A useful property}
For completeness, we state and prove a property which is used repeatedly in the paper, in particular in the proofs of Lemmas  \ref{lem:2}, \ref{lem:3},
 \ref{lem:182} and  \ref{lem:183}.
\begin{lemma}\label{lemma_appendix}
  Let $H:\R^d\times\R^d\to \R$ be a continuous Hamiltonian, convex with respect to its second argument. Consider two locally  Lipschitz continuous viscosity subsolutions $u_i$, $i=1,2$, of
$\lambda v + H(x, Dv)\le 0$ in $\R^d$. Then $u=\min(u_1, u_2) $ is also a subsolution.
\end{lemma}

\begin{proof}

We know that for $i=1,2$, $\lambda u_i(x) + H(x, Du_i(x))\le 0$
at all $x\in \R^d\setminus E$, where $E$ is a negligible set,
 see also \cite[Prop 1.9, Chapter 1, page 31]{MR1484411}.
  
We also observe that $u$ is locally Lipschitz continuous as the minimum of two locally Lipschitz continuous functions. Since $u$, $u_1$ and $u_2$ are locally Lipschitz continuous, these functions are differentiable at almost every $x\in \R^d$.

Set $A=\{x\in \R^d \setminus E:  u, u_1, u_2  \hbox{ are  differentiable at } x\}$. The set $\R^d \setminus A $ is negligible.

Consider first $x\in A$ such that $u(x)=u_1(x)< u_2(x)$. There exists $r>0$ such that $u(y)=u_1(y)< u_2(y)$ in $B_r(x)$, hence $ Du(x)=Du_1(x)$ and  $\lambda u(x) + H(x, Du(x))\le 0$.

The same argument can be applied if  $u(x)=u_2(x)< u_1(x)$. In this case,  $ Du(x)=Du_2(x)$ and  $\lambda u(x) + H(x, Du(x))\le 0$.

Next, let us prove that it is not possible to find $x\in A$ such that $u_1(x)=u_2(x)$ and $Du_1(x)\not= Du_2(x)$. If it was the case, then the sets  $F_{-}= \{\xi\in \R^ d:  \xi \cdot  (Du_1(x)-Du_2(x)) <0\}$ and  $F_{+}=-F_-= \{\xi\in \R^ d:  \xi \cdot  (Du_1(x)-Du_2(x)) >0\}$ would be  open half vector spaces.
There would exists  $r>0$ such that
\begin{equation*}
  \begin{split}
    u_1(y)&=u_1(x)+Du_1(x)\cdot(y-x)+ o(|y-x|)\\
    &= u_2(x)+Du_1(x)\cdot(y-x)+ o(|y-x|)<
    u_2(y)    
  \end{split}
\end{equation*}
in $(\{x\}+F_-)\cap B_r(x)$, and
\begin{equation*}
  \begin{split}
  u_2(y)&=u_2(x)+Du_2(x)\cdot(y-x)+ o(|y-x|)\\&= u_1(x)+Du_2(x)\cdot(y-x)+ o(|y-x|)<
  u_1(y)
\end{split}
\end{equation*}
in $(\{x\}+F_+)\cap B_r(x)$.
Hence, $u$ would coincide with $ u_1$ in  $(\{x\}+F_-)\cap B_x(r)$ and with $u_2$ in  $(\{x\}+F_+)\cap B_r(x)$. But, since $u$ is differentiable at $x$, this would imply that for all $\xi \in F_-$
\begin{displaymath}
  Du_1(x)\cdot\xi=  Du(x)\cdot\xi, \quad \hbox{and}  -Du_2(x)\cdot\xi=  -Du(x)\cdot\xi,
\end{displaymath}
and therefore $Du(x)=Du_1(x)=Du_2(x)$, which is the desired contradiction.


We have proved that for $x\in A$, $\lambda u(x) + H(x, Du(x))\le 0$, so the inequality is
satisfied pointwise almost everywhere. But $H$ is convex in $p$, so thanks to \cite[Prop 5.1, Chapter 1, page 77]{MR1484411}, $u$ is a viscosity subsolution of $\lambda v + H(x, Dv)\le 0$ in $\R^d$.

\begin{remark}
  Note that if $H$ is coercive in $p$ uniformly in $x$, then the viscosity subsolutions of $\lambda v + H(x, Dv)\le 0$ which are (locally) bounded are  (locally) Lipschitz continuous.
\end{remark}

\end{proof}

 \section*{Acknowledgments}
  The first author was  partially supported by the ANR (Agence Nationale de la Recherche) through project ANR-16-CE40-0015-01 and  by the chair Finance and Sustainable Development and FiME Lab (Institut Europlace de Finance).   The first author would like to thank Nicoletta Tchou for quite helpful discussions.

\bibliographystyle{siam}
\bibliography{Hom_HJB_defect}

\begin{thebibliography}{10}

\bibitem{MR4643677}
{\sc Y.~Achdou and C.~Le~Bris}, {\em Homogenization of some periodic
  {H}amilton-{J}acobi equations with defects}, Comm. Partial Differential
  Equations, 48 (2023), pp.~944--986.

\bibitem{MR3358634}
{\sc Y.~Achdou, S.~Oudet, and N.~Tchou}, {\em Hamilton-{J}acobi equations for
  optimal control on junctions and networks}, ESAIM Control Optim. Calc. Var.,
  21 (2015), pp.~876--899.

\bibitem{MR3565416}
\leavevmode\vrule height 2pt depth -1.6pt width 23pt, {\em Effective
  transmission conditions for {H}amilton-{J}acobi equations defined on two
  domains separated by an oscillatory interface}, J. Math. Pures Appl. (9), 106
  (2016), pp.~1091--1121.

\bibitem{MR3299352}
{\sc Y.~Achdou and N.~Tchou}, {\em Hamilton-{J}acobi equations on networks as
  limits of singularly perturbed problems in optimal control: dimension
  reduction}, Comm. Partial Differential Equations, 40 (2015), pp.~652--693.

\bibitem{MR3912640}
\leavevmode\vrule height 2pt depth -1.6pt width 23pt, {\em Homogenization of a
  transmission problem with {H}amilton-{J}acobi equations and a two-scale
  interface. {E}ffective transmission conditions}, J. Math. Pures Appl. (9),
  122 (2019), pp.~164--197.

\bibitem{MR1484411}
{\sc M.~Bardi and I.~Capuzzo-Dolcetta}, {\em Optimal control and viscosity
  solutions of {H}amilton-{J}acobi-{B}ellman equations}, Systems \& Control:
  Foundations \& Applications, Birkh\"auser Boston Inc., Boston, MA, 1997.
\newblock With appendices by Maurizio Falcone and Pierpaolo Soravia.

\bibitem{barles2013bellman}
{\sc G.~Barles, A.~Briani, and E.~Chasseigne}, {\em A {B}ellman approach for
  regional optimal control problems in {$\mathbb{R}^N$}}, SIAM J. Control
  Optim., 52 (2014), pp.~1712--1744.

\bibitem{MR4704058}
{\sc G.~Barles and E.~Chasseigne}, {\em On modern approaches of
  {H}amilton-{J}acobi equations and control problems with discontinuities---a
  guide to theory, applications, and some open problems}, vol.~104 of Progress
  in Nonlinear Differential Equations and their Applications,
  Birkh\"{a}user/Springer, Cham, 2024.
\newblock PNLDE Subseries in Control.

\bibitem{Milan}
{\sc X.~Blanc, C.~Le~Bris, and P.-L. Lions}, {\em A possible homogenization
  approach for the numerical simulation of periodic microstructures with
  defects}, Milan J. Math., 80 (2012), pp.~351--367.

\bibitem{CPDE-2015}
\leavevmode\vrule height 2pt depth -1.6pt width 23pt, {\em Local profiles for
  elliptic problems at different scales: defects in, and interfaces between
  periodic structures}, Comm. Partial Differential Equations, 40 (2015),
  pp.~2173--2236.

\bibitem{CPDE-2018}
\leavevmode\vrule height 2pt depth -1.6pt width 23pt, {\em On correctors for
  linear elliptic homogenization in the presence of local defects}, Comm.
  Partial Differential Equations, 43 (2018), pp.~965--997.

\bibitem{MR2291823}
{\sc A.~Bressan and Y.~Hong}, {\em Optimal control problems on stratified
  domains}, Netw. Heterog. Media, 2 (2007), pp.~313--331.

\bibitem{MR951880}
{\sc I.~Capuzzo-Dolcetta and P.-L. Lions}, {\em Hamilton-{J}acobi equations
  with state constraints}, Trans. Amer. Math. Soc., 318 (1990), pp.~643--683.

\bibitem{MR4450245}
{\sc N.~El~Khatib, N.~Forcadel, and M.~Zaydan}, {\em Homogenization of a
  microscopic pedestrians model on a convergent junction}, Math. Model. Nat.
  Phenom., 17 (2022), pp.~Paper No. 21, 37.

\bibitem{MR1007533}
{\sc L.~C. Evans}, {\em The perturbed test function method for viscosity
  solutions of nonlinear {PDE}}, Proc. Roy. Soc. Edinburgh Sect. A, 111 (1989),
  pp.~359--375.

\bibitem{MR3809148}
{\sc N.~Forcadel, W.~Salazar, and M.~Zaydan}, {\em Specified homogenization of
  a discrete traffic model leading to an effective junction condition}, Commun.
  Pure Appl. Anal., 17 (2018), pp.~2173--2206.

\bibitem{MR3441209}
{\sc G.~Galise, C.~Imbert, and R.~Monneau}, {\em A junction condition by
  specified homogenization and application to traffic lights}, Anal. PDE, 8
  (2015), pp.~1891--1929.

\bibitem{MR3621434}
{\sc C.~Imbert and R.~Monneau}, {\em Flux-limited solutions for quasi-convex
  {H}amilton-{J}acobi equations on networks}, Ann. Sci. \'Ec. Norm. Sup\'er.
  (4), 50 (2017), pp.~357--448.

\bibitem{MR3690310}
\leavevmode\vrule height 2pt depth -1.6pt width 23pt, {\em Quasi-convex
  {H}amilton-{J}acobi equations posed on junctions: the multi-dimensional
  case}, Discrete Contin. Dyn. Syst., 37 (2017), pp.~6405--6435.

\bibitem{PLL-college}
{\sc P.-L. Lions}, {\em {\'E}quations elliptiques ou paraboliques, et
  homog{\'e}n{\'e}isation pr{\'e}cis{\'e}e [{E}lliptic or parabolic equations
  and precised homogenization], {L}ectures at {C}oll{\`e}ge de {F}rance
  2013/14}.
\newblock
  text{https://www.college-de-france.fr/fr/agenda/cours/equations-elliptiques-ou-paraboliques-et-homogeneisation-precisee}.

\bibitem{LPV}
{\sc P.-L. Lions, G.~Papanicolaou, and S.~Varadhan}, {\em Homogenization of
  {H}amilton-{J}acobi equations}.
\newblock unpublished, circa 1988.

\bibitem{PLL}
{\sc P.-L. Lions and P.~Souganidis}, {\em in {L}ectures at {C}oll{\`e}ge de
  {F}rance, see~\cite{PLL-college}}.

\bibitem{MR3556345}
\leavevmode\vrule height 2pt depth -1.6pt width 23pt, {\em Viscosity solutions
  for junctions: well posedness and stability}, Atti Accad. Naz. Lincei Rend.
  Lincei Mat. Appl., 27 (2016), pp.~535--545.

\bibitem{oudet2015equations}
{\sc S.~Oudet}, {\em {\'E}quations de Hamilton-Jacobi sur des r{\'e}seaux ou
  des structures h{\'e}t{\'e}rog{\`e}nes}, PhD thesis, U. Rennes 1, 2015.
\newblock
  {https://ged.univ-rennes1.fr/nuxeo/site/esupversions/d8b3c38f-e558-4025-9314-eb971daec600?inline}.

\bibitem{MR838056}
{\sc H.~M. Soner}, {\em Optimal control with state-space constraint. {I}}, SIAM
  J. Control Optim., 24 (1986), pp.~552--561.

\bibitem{MR4523487}
{\sc S.~Wolf}, {\em Homogenization of the {$p$}-{L}aplace equation in a
  periodic setting with a local defect}, Nonlinear Anal., 228 (2023), pp.~Paper
  No. 113182, 38.

\end{thebibliography}

\end{document}